\documentclass[a4paper,11pt]{amsart}
\usepackage[english]{babel}
\usepackage{amsmath,amsthm,amssymb,amsfonts}
\usepackage{multicol}

\usepackage[pdftex]{color}
\usepackage[bookmarks=true,hyperindex,pdftex,colorlinks, citecolor=blue,linkcolor=blue, urlcolor=blue
]{hyperref}
\usepackage[dvipsnames]{xcolor}
\usepackage{mathtools}
\usepackage{graphicx}
\usepackage[titletoc]{appendix}
\usepackage{tikz}

\usepackage[normalem]{ulem}

\usetikzlibrary{matrix}

\newtheorem{theorem}{Theorem}[section]
\newtheorem{lemma}[theorem]{Lemma}
\newtheorem{fact}[theorem]{Fact}

\newtheorem{proposition}[theorem]{Proposition}
\newtheorem{corollary}[theorem]{Corollary}

\theoremstyle{definition}
\newtheorem{definition}[theorem]{Definition}
\newtheorem{example}[theorem]{Example}

\theoremstyle{remark}
\newtheorem{remark}[theorem]{Remark}

\numberwithin{equation}{section}

\newcommand{\R}{\mathbb{R}}
\newcommand{\C}{\mathbb{C}}
\newcommand{\N}{\mathbb{N}}
\newcommand{\K}{\mathbb{K}}

\newcommand{\T}{\mathbb{T}}

\DeclareMathOperator{\co}{co}
\DeclareMathOperator{\e}{e}
\DeclareMathOperator{\re}{Re}

\DeclareMathOperator{\id}{Id}
\newcommand{\cconv}{\overline{\co}}

\newcommand{\nn}[1]{{\left\vert\kern-0.25ex\left\vert\kern-0.25ex\left\vert #1
		\right\vert\kern-0.25ex\right\vert\kern-0.25ex\right\vert}}

\renewcommand{\geq}{\geqslant}
\renewcommand{\leq}{\leqslant}
\renewcommand{\flat}{\mathrm{Fl}_\Gamma}

\newcommand{\Free}{{\mathcal F}}
\newcommand{\lip}{{\mathrm{lip}}_0}

\newcommand{\Lip}{{\mathrm{Lip}}_0}
\newcommand{\Lin}{\mathcal{L}}
\newcommand{\Linw}{\Lin_{w^*-w^*}}
\newcommand{\cpt}{\mathcal{K}}
\newcommand{\ASE}{\operatorname{ASE}}
\newcommand{\SE}{\operatorname{SE}}

\newcommand{\SA}{\operatorname{SNA}}
\newcommand{\NA}{\operatorname{NA}}

\newcommand{\F}[1]{\mathcal{F}(#1)}
\newcommand{\aconv}{\mathop\mathrm{aconv}}

\newcommand{\ext}[1]{\operatorname{ext}(#1)}
\newcommand{\strexp}[1]{\operatorname{str-exp}(#1)}
\newcommand{\wstrexp}[1]{\ensuremath{w^*}\operatorname{-str-exp}(#1)}

\newcommand{\Mol}[1]{\operatorname{Mol}\left(#1\right)}

\newcommand{\dist}{\operatorname{dist}}

\newcommand{\pten}{\ensuremath{\widehat{\otimes}_\pi}}

\newcommand{\eps}{\varepsilon}

\begin{document}

\title[Norm attaining operators]{Residuality in the set of norm attaining operators between Banach spaces}
\date{March 8th, 2022}

\author[Jung]{Mingu Jung}
\address[Jung]{School of Mathematics, Korea Institute for Advanced Study, 02455 Seoul, Republic of Korea \newline
\href{http://orcid.org/0000-0000-0000-0000}{ORCID: \texttt{0000-0003-2240-2855} }}
\email{\texttt{jmingoo@kias.re.kr}}
\urladdr{\url{https://clemg.blog/}}

\author[Mart\'in]{Miguel Mart\'in}
\address[Mart\'{\i}n]{Universidad de Granada, Facultad de Ciencias.
Departamento de An\'{a}lisis Matem\'{a}tico, 18071 Granada, Spain \newline
	\href{http://orcid.org/0000-0003-4502-798X}{ORCID: \texttt{0000-0003-4502-798X} }}
\email{mmartins@ugr.es}
\urladdr{\url{https://www.ugr.es/local/mmartins}}

\author[Rueda Zoca]{Abraham Rueda Zoca}
\address[Rueda Zoca]{Universidad de Murcia, Departamento de Matem\'aticas, Campus de Espinardo 30100 Murcia, Spain
 \newline
\href{https://orcid.org/0000-0003-0718-1353}{ORCID: \texttt{0000-0003-0718-1353} }}
\email{\texttt{abraham.rueda@um.es}}
\urladdr{\url{https://arzenglish.wordpress.com}}

\begin{abstract}
We study the relationship between the residuality of the set of norm attaining functionals on a Banach space and the residuality and the denseness of the set of norm attaining operators between Banach spaces. Our first main result says that if $C$ is a bounded subset of a Banach space $X$ which admit an LUR renorming satisfying that, for every Banach space $Y$, the operators $T$ from $X$ to $Y$ for which the supremum of $\|Tx\|$ with $x\in C$ is attained are dense, then the $G_\delta$ set of those functionals which strongly exposes $C$ is dense in $X^*$. This extends previous results by J.\ Bourgain and K.-S.\ Lau. The particular case in which $C$ is the unit ball of $X$, in which we get that the norm of $X^*$ is Fr\'{e}chet differentiable at a dense subset, improves a result by J.\ Lindenstrauss and we even present an example showing that Lindenstrauss' result was not optimal.
In the reverse direction, we obtain results for the density of the $G_\delta$ set of absolutely strongly exposing operators from $X$ to $Y$ by requiring that the set of strongly exposing functionals on $X$ is dense and conditions on $Y$ or $Y^*$ involving RNP and discreteness on the set of strongly exposed points of $Y$ or $Y^*$. These results include examples in which even the denseness of norm attaining operators was unknown.
We also show that the residuality of the set of norm attaining operators implies the denseness of the set of absolutely strongly exposing operators provided the domain space and the dual of the range space are separable, extending a recent result for functionals.
Finally, our results find important applications to the classical theory of norm-attaining operators, to the theory of norm-attaining bilinear forms, to the geometry of the preduals of spaces of Lipschitz functions, and to the theory of strongly norm-attaining Lipschitz maps. In particular, we solve a proposed open problem showing that the unique predual of the space of Lipschitz functions from the Euclidean unit circle fails to have Lindenstrauss property A.
\end{abstract}


\maketitle

\thispagestyle{plain}

\setcounter{tocdepth}{1}
\tableofcontents

\setcounter{tocdepth}{2}

\section{Introduction}
Given Banach spaces $X$ and $Y$, we write $\Lin(X,Y)$ to denote the space of all (bounded linear) operators from $X$ to $Y$, and $\NA(X, Y)$ for the set of norm attaining operators (i.e., those $T\in \Lin(X,Y)$ for which there is a norm one $x\in X$ such that $\|T\|=\|Tx\|$). Our notation is standard, and it can be found in Subsection \ref{subsection:notation} together with the definition of some needed well known concepts. The study of the denseness of the set of norm attaining operators started with the celebrated result by Bishop and Phelps of the 1960's that $\NA(X,\K)$ is dense in $X^*\equiv \Lin(X,\K)$ for every Banach space $X$ ($\K$ denotes the base field $\R$ or $\C$). Shortly afterward, J.~Lindenstrauss initiated a systemic study on norm attaining operators between Banach spaces \cite{Lin}. He introduced two properties -- nowadays called (Lindenstrauss) properties A and B -- as follows: a Banach space $X$ has \emph{property A} if $\NA(X, W)$ is dense in $\Lin (X, W)$ for every Banach space $W$, and a Banach space $Y$ has \emph{property B} if $\NA(Z,Y)$ is dense in $\Lin(Z,Y)$ for every Banach space $Z$. What he proved is, among other results, that reflexive spaces and those spaces for which the unit ball is the closed convex hull of a set of uniformly strongly exposed points, have property A. It is also shown that finite-dimensional spaces whose dual unit ball have finitely many extreme points up to rotations (in the real case, these are finite-dimensional spaces whose unit ball is a polyhedron) and subspaces of $\ell_\infty$ containing the canonical copy of $c_0$ (among other spaces) have property B. On the other hand, Lindenstrauss presented a useful necessary condition for Banach spaces to have property A. Namely, if a Banach space admits an LUR renorming and has property A, then its closed unit ball is the closed convex hull of its strongly exposed points. Up to our knowledge, this is the strongest necessary condition for property A which has appeared in the literature.

In 1977, J.\ Bourgain linked the study of the denseness of norm attaining operators with the Radon-Nikod\'{y}m property (RNP, for short) in the remarkable paper \cite{Bou}. It is shown that the RNP implies property A and, conversely, that if a Banach space $X$ has property A for all equivalent norms, then $X$ has the RNP (this formulation requires a refinement made by R. Huff \cite{Huff}). Since then, there has been an intensive research on norm attaining operators, an account of which can be found in the expository papers \cite{Aco-survey, Acosta-survey-BPBP, Martin2016}. Let us just mention here a few known results. With respect to property A, apart from the aforementioned results on the RNP, it is known that any weakly compactly generated space can be renormed with property A \cite{Scha1983}; examples of Banach spaces failing property A in their usual norm are $C_0(L)$ for infinite metrizable space $L$ and $L_1(\mu)$ when $\mu$ is not purely atomic \cite{Lin}. With respect to property B, it is known that infinite dimensional strictly convex Banach spaces fail property B and the same happens with any infinite-dimensional $L_1 (\mu)$  \cite{Acosta1999}, and that every Banach space can be renormed to have property B \cite{Partington}. Moreover, there are even compact operators which can not be approximated by norm attaining ones \cite{martinjfa}. It is not known, however, whether finite rank operators can be always approximated by norm attaining operators.

The main importance of Bourgain's paper \cite{Bou} is that it relates (via the concept of RNP) two geometric properties whose relationship was unknown at that moment: dentability and strong exposition. Indeed, Bourgain actually studied the following generalization of Lindenstrauss property A replacing the unit ball with a bounded closed convex subset. A bounded subset $C$ of a Banach space $X$ has the \emph{Bishop-Phelps property} if, for every Banach space $Y$, the set of those operators in $\Lin(X,Y)$ for which $\sup \{\|Tx\|\colon x\in C\}$ is a maximum, is dense in $\Lin(X,Y)$. A Banach space $X$ has the \emph{Bishop-Phelps property} if all of its bounded closed absolutely convex subsets have the property. What Bourgain proved is that a Banach space has the Bishop-Phelps property if and only if it has the RNP. This equivalence is proved through the following two separate results:
\begin{itemize}
\setlength\itemsep{0.3em}
 \item[(a)] If $C$ is a separable bounded closed convex set with the Bishop-Phelps property, then it is dentable (i.e.,\ it contains slices of arbitrarily small diameter).
 \item[(b)] If $B$ is a bounded closed absolutely convex subset of a Banach space $X$ such that every nonempty subset of $B$ is dentable (that is, $B$ is an \emph{RNP set}), then $B$ has the Bishop-Phelps property. Moreover, for every Banach space $Y$, the set of those operators in $\Lin(X,Y)$ which absolutely strongly expose $B$ is dense in $\Lin(X,Y)$.
\end{itemize}

The first main aim of this paper is to give an improvement of the item (a) above, which is presented in Section \ref{sec:BPandpropertyA}. Indeed, Theorem~\ref{thm:LURrenorming} shows that for a bounded subset $C$ of a Banach space $X$ admitting an LUR renorming and having the Bishop-Phelps property, the set of its strongly exposing functionals is dense in $X^*$ so, in particular, $C$ is contained in the closed convex hull of its strongly exposed points. This result generalized the same conclusion already known for weakly compact convex sets \cite{Bou2,Lau} and for bounded closed convex RNP sets \cite[Theorem~8]{Bou}. The particular case of Theorem \ref{thm:LURrenorming} for the unit ball of a Banach space $X$ (Corollary~\ref{corollary:PropertyAimpliesSEdense}) provides a necessary condition for Lindenstrauss property A stronger than the one given in \cite{Lin}: if $X$ admits an LUR renorming and has property A, then the norm of $X^*$ is Fr\'{e}chet differentiable at a dense subset of $X^*$ (or, equivalently, the set $\SE(X)$ of strongly exposing functionals on $X$ is dense in $X^*$). This allows us to show that Lindenstrauss' original necessary condition for property A is not sufficient, see Example~\ref{example:Lindenstrauss-not-sufficient} and Remark~\ref{remark:torusnotA}.

With respect to the result in item (b) above, let us give some comments. First, this result was extended from far by C.~Stegall \cite{Stegall1986} to a wide class of non-linear functions defined on bounded RNP sets, which is now known as the Stegall variational principle. Second, Bourgain's result is stronger than the mere denseness of norm attaining operators, even when the set $B$ is the unit ball of a Banach space. On the one hand, it provides the denseness of operators $T$ such that the application $x\longmapsto \|Tx\|$ attains an strong maximum (up to rotations). On the other hand, as the set of absolutely strongly exposing operators is a $G_\delta$ set, his result shows that $\NA(X,Y)$ is residual in $\Lin(X,Y)$. Some consequences of the residuality of the set of norm attaining operators are included in subsection~\ref{subsect:residuality-easy-consequences}. Let us comment that other topological properties of the set of norm attaining functionals and norm attaining operators have been studied in the literature, see \cite{DebsGodefroySaint} for functionals and \cite{Bachir} for operators, for instance.

In Section \ref{section:sufficient}, we investigate the possible density of the set $\ASE (X,Y)$ of absolutely strongly exposing operators from $X$ to $Y$, which is our second main aim in this paper. It is easy to show that the denseness of $\ASE(X,Y)$ in $\Lin(X,Y)$ for a non-trivial Banach space $Y$ implies that $\SE(X)$ is dense in $X^*$ (see Proposition~\ref{prop:goingdown}). We do not know when the converse result holds, so the following is our leading question.
\begin{equation}
  \tag{Q1}\label{Q1}
  \parbox{\dimexpr\linewidth-4em}{%
    \strut
\slshape Find conditions on $Y$ such that $\ASE (X,Y)$ is dense in $\Lin (X,Y)$ whenever $\SE(X)$ is dense in $X^*$.
    \strut
  }
\end{equation}
It is well known that $\SE(X)$ is dense in $X^*$ when $X$ has the RNP and in this case $\ASE(X,Y)$ is dense for every $Y$ thanks to (b) above. On the other hand, $\SE(X)$ is also dense in $X^*$ if the Banach space
$X$ is ALUR (in particular, if $X$ is LUR) or even when every element in the unit sphere of $X$ is strongly exposed (see Proposition \ref{prop:ALUR} where a stronger result is proved). Let us comment on this that we do not know whether being LUR implies property A, so in this case partial answers to \eqref{Q1} are especially interesting.

Some of our main results in Section~\ref{section:sufficient} are the following ones. Let $X$ be a Banach space for which $\SE(X)$ is dense in $X^*$. Then, the $G_\delta$ set $\ASE(X,Y)$ is dense in $\Lin(X,Y)$ provided the range space $Y$ is in one of the situations below:
\begin{enumerate}
\setlength\itemsep{0.3em}
\item $Y$ has property quasi-$\beta$ (Theorem \ref{thm:quasibeta}),
\item $Y$ has ACK$_\rho$ structure and $X$ or $Y$ are Asplund (Corollary~\ref{corollary:ACKrho-structure-XorY-Asplund}),
\item $Y$ has the RNP and $\strexp{B_Y}$ is either countable up to rotations or discrete up to rotations (Theorem \ref{coro:countably-1} and \ref{thm:RNP}),
\item $Y^*$ has the RNP and $\strexp{B_{Y^*}}$ is countable up to rotations (Theorem \ref{coro:countably-2}),
\item $Y^*$ has the RNP and for every sequence $\{y_n^*\}$ of elements of $\wstrexp{B_{Y^*}}$ which converges to an element $y_0^*\in \strexp{B_{Y^*}}$, there exist $n_0\in \N$ and a sequence $\{\theta_n\}$ in $\T$ such that $y_n^*=\theta_n y_0^*$ for every $n\geq n_0$ (Theorem \ref{thm:dualRNP}).
\end{enumerate}
We also give several concrete examples where the above result applies, including preduals of $\ell_1(\Gamma)$ spaces and finite-dimensional spaces for which the dual unit ball has countably many extreme points (see Example \ref{Example:quasibeta_example}, \ref{example:finite_dimensional_countable}, \ref{example:ell_1_gamma}, and Remark \ref{rem:appliACKstructure}). For the cases (3), (4), and (5), even the denseness of $\NA(X,Y)$ was unknown for many $X$s. Let us also mention that in Example \ref{Example:quasibeta_example}, new examples of Banach spaces having property quasi-$\beta$ (hence Lindenstrauss property B), such as real polyhedral predual spaces of $\ell_1$ and arbitrary (real or complex) closed subspaces of $c_0 (\Gamma)$, are exhibited.

By (b) above, $\ASE(X,Y)$ is dense in $\Lin(X,Y)$ for any arbitrary Banach space $Y$ when $X$ has the RNP. Besides, it was observed in \cite[Proposition~4.2]{cgmr2020} that for $X$ satisfying any of the known conditions which guarantee property A (namely, property $\alpha$, property quasi-$\alpha$, or having a norming subset of uniformly strongly exposed points), the set $\ASE(X,Y)$ is dense in $\Lin(X,Y)$ for every Banach space $Y$. We do not know, however, whether property A of $X$ implies the denseness of $\ASE(X,Y)$ for all Banach spaces $Y$.
\begin{equation}
  \tag{Q2}\label{Q2}
  \parbox{\dimexpr\linewidth-4em}{%
    \strut
\slshape Does property A of a Banach space $X$ imply that $\ASE (X,Y)$ is dense in $\Lin (X,Y)$ for every Banach space $Y$?
    \strut
  }
\end{equation}
Corollary~\ref{corollary:PropertyAimpliesSEdense} links \eqref{Q2} with the previous question \eqref{Q1} and allows us to present partial answers to the question \eqref{Q2} by applying the aforementioned results. Namely, if $X$ has property A and admits an equivalent LUR renorming and $Y$ satisfies one of the conditions (1)--(5) above, then $\ASE(X,Y)$ is dense in $\Lin (X, Y)$.

Furthermore, we obtain some results concerning the denseness of the set $\ASE (X,Y) \cap \mathcal{K} (X,Y)$, where $\mathcal{K} (X,Y)$ denotes the space of all compact linear operators from $X$ to $Y$.
We prove that the denseness of $\SE(X)$ in $X^*$ implies $\ASE (X,Y) \cap \mathcal{K} (X,Y)$ is dense in $\mathcal{K} (X,Y)$ not only when the Banach space $Y$ is in one of the above conditions (1), (3), (4), (5), but also when $Y$ is an $L_1$-predual, or has ACK$_\rho$ structure (this is (2) without the Asplundness condition on $X$ or $Y$), or it admits a countable James boundary (see Example \ref{thm:cpt}, Theorem \ref{propo:Gammaflat-just-denseness} and Corollary \ref{coro:countableboundary-compact}). In particular, if $Y$ is a (real) polyhedral space or $Y$ is a closed subspace of a $C(K)$ space for an scattered Hausdorff compact topological space $K$ (see Examples \ref{example:General_polyhedralcompact} and \ref{example:scattered}).

In the fourth section of the paper, we prove that the residuality of $\NA(X,Y)$ in $\Lin (X,Y)$ is equivalent to the denseness of the set of points of $\Lin (X,Y)$ at which the norm is Fr\'echet differentiable provided $X$ and $Y^*$ are separable Banach spaces, which generalizes a result of Guirao, Montesinos, and Zizler \cite[Theorem 3.1]{GMZ} (see Theorem \ref{thm:residual}). Moreover, by using a result of Moors and Tan \cite{MoorsTan} and one of Avil\'es et al.\ \cite{AMRR}, we observe that this equivalence also holds in the case when $X$ is a subspace of weakly compactly generated space, $Y$ is a reflexive space, and $\Lin (X,Y) = \mathcal{K}(X,Y)$ (see Remark \ref{remark:residual}).

In the final section, we present several applications of the results in Section \ref{sec:BPandpropertyA} and \ref{section:sufficient} to the geometry of Lipschitz free spaces, to the denseness of strongly norm attaining Lipschitz maps, and to the denseness of strongly norm attaining bilinear forms. First, we show that for a separable metric space $M$, property A of $\mathcal{F}(M)$ forces the density of the set of strongly norm attaining Lipschitz functions on $M$ (Corollary \ref{coro:corLipschitzfree}). As a consequence, we show that the Banach space $\mathcal{F}(\T)$ fails to have property A (Example \ref{example:Lindenstrauss-not-sufficient}), answering a question implicitly possed in \cite{cgmr2020}. We also present new examples of Banach spaces, coming from the theory of Lipschitz maps, which can be used as target spaces in the results of Section~\ref{section:sufficient} (Examples \ref{exam:extremeLipfree1}, \ref{exam:extremeLipfree3}, and \ref{exam:extremeLipfree2}). Next, some sufficient conditions on a metric space $M$ and on a Banach space $Y$ are discussed for the set of strongly norm attaining Lipschitz maps from $M$ into $Y$ to be dense (Corollary \ref{coro:SNAdensefromA}). Finally, some results on the density of strongly norm attaining bilinear forms are presented (Corollaries  \ref{cor:ChoiSong} and \ref{cor:bilinearRNP}) which improve previous results.

The rest of this introduction is devoted to introduce the needed notation and preliminaries (Subsection~\ref{subsection:notation}), to present some background on absolutely strongly exposing operators (Subsection~\ref{subsection:preminiaryASE}), and to expose some consequences of residuality of norm attaining operators (Subsection~\ref{subsect:residuality-easy-consequences}).

\subsection{Notation and preliminaries}\label{subsection:notation}
Here $\K$ denotes the field $\R$ of real numbers or $\C$ of complex numbers, and $\T$ is the subset of $\K$ of modulus one elements. Let $X$ and $Y$ be Banach spaces over $\K$. We write $B_X$ and $S_X$ to denote, respectively, the closed unit ball and the unit sphere of $X$. Given $x\in X$ and $r>0$, $B(x,r)$ is the open ball centered in $x$ with radius $r$.

The notation $\Lin(X,Y)$ stands for the space of all bounded linear operators from $X$ to $Y$ and we simply write $X^*\equiv \Lin(X,\K)$. We write $\Linw (Y^*, X^*)=\{T^*\colon T\in \Lin(X,Y)\}$ for the space of $w^*$-$w^*$-continuous bounded linear operators from $Y^*$ into $X^*$ which is isometrically isomorphic to $\Lin(X,Y)$. The space of all compact linear operators from $X$ to $Y$ is denoted by $\mathcal{K}(X,Y)$ and $X\pten Y$ denotes the projective tensor product of $X$ and $Y$.

For a nonempty bounded subset $C$ of $X$, a point $x_0 \in C$ is called an \emph{exposed point of $C$} if there is $x^* \in X^*$ such that
$$
\re x^*(x_0) = \sup_{x \in C} \re x^* (x)\ \ \text{ and } \ \ \{ x \in C \colon \re x^* (x) = \re x^* (x_0)\} = \{x_0\}.
$$
In this case, we say that $x^*$ \emph{exposes} $x_0$ and that $x^*$ is an \emph{exposing} functional. A point $x_0 \in C$ is called a \emph{strongly exposed point of $C$} if there is $x^* \in X^*$ such that $\re x^*(x_0) = \sup_{x \in C} \re x^* (x)$ and $\{x_n\}$ converges in norm to $x_0$ for all sequences $\{x_n\} \subseteq C$ such that $\lim_n \re x^* (x_n) = x^* (x_0)$. In this case, we say that $x^*$ \emph{strongly exposes} $x_0$ in $C$ and $x^*$ is said to be a \emph{strongly exposing functional} of $C$. We write $\strexp{C}$ and $\SE(C)$ for, respectively, the set of strongly exposed points of $C$ and the set of strongly exposing functionals of $C$. It is immediate that $\R^+\SE(C)=\SE(C)$; if $C$ is actually balanced (i.e.\ $\lambda C=C$ for every $\lambda\in \K$ with $|\lambda|=1$), then $\lambda\SE(C)=\SE(C)$ for every $\lambda\in \K\setminus\{0\}$. In the case $C=B_X$, we just write $\SE(X):=\SE(B_X)$ and call strongly exposing functionals to their elements. It is well known that $x^*\in \SE(X)$ if and only if the norm of $X^*$ is Fr\'{e}chet differentiable at $x^*$ (\v{S}mulyan test, see \cite[Corollary 1.5]{DGZ} for instance). If $X=Z^*$ is a dual space and $z^*\in B_{Z^*}$ is strongly exposed by some $x\in X\subset X^{**}$, we say that $z^*$ is a \emph{$w^*$-strongly exposed point}.

A point $x_0\in X$ is said to be a \emph{locally uniformly rotund point} (LUR point, for short) if whenever $\{x_n\}$ is a sequence in $X$ such that $\|x_n\|\leq \|x_0\|$ for every $n\in \N$ and $\|x_n+x_0\|\longrightarrow 2\|x_0\|$, then $\|x_n-x_0\|\longrightarrow 0$. A Banach space $X$ is \emph{locally uniformly rotund} (LUR, for short) if all the elements in $X$ are LUR points, equivalently, if all elements in $S_X$ are LUR points. It is known that weakly compactly generated (for short, WCG) Banach spaces (in particular, separable or reflexive Banach spaces) admit LUR equivalent renormings \cite[Theorem 1]{Troy}. A point $x_0\in X$ is said to be a \emph{rotund point} if for every $x\in X$ with $\|x\|\leq \|x_0\|$ and $\|x+x_0\|=2\|x_0\|$, we have that $x=x_0$. A Banach space $X$ is \emph{rotund} if all the elements in $X$ are rotund points, equivalently, if all elements in $S_X$ are rotund points, equivalently, if all elements in $S_X$ are extreme points. It is known that $\ell_\infty$ admits a rotund equivalent norm \cite[Theorem 8.13]{checos} but not an LUR equivalent norm (this follows since $\ell_\infty$ does not admit any equivalent norm with the Kadec-Klee property \cite[Theorem II.7.10]{DGZ}).

\subsection{Some preliminary results on absolutely strongly exposing operators}\label{subsection:preminiaryASE}
Let $X$, $Y$ be Banach spaces and let $B\subset X$ be a bounded closed absolutely convex subset.
An operator $T \in \Lin (X,Y)$ is said to \emph{absolutely strongly exposes} $B$ if there exists $x_0 \in B$ such that whenever a sequence $\{x_n\}$ in $B$ satisfies $\|T x_n \| \longrightarrow \sup\{\|Tx\|\colon x\in B\}$, then there exists a sequence $\{\theta_n\}$ of elements of $\T$ such that $\{\theta_n x_n\}\longrightarrow x_0$. When $B=B_X$, we just say that $T$ is an \emph{absolutely strongly exposing operator} and write $\ASE(X,Y)$ for the set of those operators. This is the case that we are most interested in. It is easy to see and well known that $\ASE(X,Y)$ is a $G_\delta$ subset of $\Lin(X,Y)$. Indeed, given $\eps>0$, we consider the subsets
$$
\mathcal{A}_\eps=\bigl\{T\in \Lin(X,Y)\colon S(T, \eta)\subset \T B(x_0,\eps)\text{ for some } x_0 \in X,\, \eta>0\bigr\}
$$
where $S(T, \eta)=\{x \in B_{X}\colon \|T(x)\| > \|T\| - \eta \}$. Then, each set $\mathcal{A}_\eps$ is open and, clearly,
$$
\ASE(X,Y)=\bigcap\nolimits_{n=1}^\infty \mathcal{A}_{r_n}
$$
for every sequence $\{r_n\}$ of positive numbers converging to $0$.

The main results on absolutely strongly exposing operators is, of course, its denseness when the domain space has the RNP (Bourgain). The next result contains two versions of this result. Item (a) follows routinely from Stegall variational principle \cite[Theorem~14]{Stegall1986} (as it is done in Theorems 15 and 19 of the same paper); item (b) follows in the same way using a weak-star version of Stegall variational principle which appeared in \cite[Theorem~2.6]{AAGM2010-variational}, and it is actually implicit in Theorem 21 of \cite{Stegall1986}.

\begin{proposition}[\mbox{Bourgain--Stegall}]\label{propo:BourgainStegall}
Let $X$ and $Y$ be Banach spaces.
\begin{enumerate}
\setlength\itemsep{0.3em}
  \item[(a)] If $X$ has the RNP, then $\ASE(X,Y)$ is residual in $\Lin(X,Y)$. Moreover, given $T\in \Lin(X,Y)$ and $\eps>0$, there is $y\in S_Y$, $x^*\in S_{X^*}$ and $0<\rho<\eps$ such that the operator $S:=T+\rho\, x^*\otimes y$ belongs to $\ASE(X,Y)$.
  \item[(b)] If $Y^*$ has the RNP, then $\ASE(Y^*,X^*)\cap \Linw(Y^*,X^*)$ is residual in $\Linw (Y^*,X^*)$. Moreover, given $T\in \Lin(X,Y)$ and $\eps>0$, there is $y\in S_Y$, $x^*\in S_{X^*}$ and $0<\rho<\eps$ such that the operator $S:=T+\rho\, x^*\otimes y$ satisfies $S^*\in \ASE(Y^*,X^*)$.
\end{enumerate}
\end{proposition}

Next, from the proof of \cite[Proposition 3.14]{cgmr2020}, we may extract the following easy results which we will use all along the paper.

\begin{lemma}[\mbox{\cite{cgmr2020}}] \label{lem:RMI}
Let $X$ and $Y$ be Banach spaces.
\begin{enumerate}
\setlength\itemsep{0.3em}
\item If $T \in \ASE(X,Y)$ with $\|Tx_0\| = \|T\|$ and $y^* \in S_{Y^*}$ satisfies that $\re y^* (Tx_0) = \|T\|$, then $T^* y^* \in \SE(X)$.
\item If $T \in \Lin (X,Y)$ attains its norm at $x_0 \in \strexp{B_X}$, then for any $\eps >0$, there exists $S \in \ASE(X,Y)$ such that $\|Sx_0\| = \|S\|$ and $\|S-T\| < \eps$. Moreover, $S-T$ is of rank one and $Sx_0\in \R^+Tx_0$.
\item If $T\in \Linw (Y^*,X^*)$ attains its norm at $y_0^*\in \wstrexp{B_{Y^*}}$, then for any $\eps>0$, there exists $S\in \ASE(Y^*,X^*)\cap \Linw (Y^*,X^*)$ such that $\|Sy_0^*\|=\|S\|$ and $\|S-T\|<\eps$. Moreover, $S-T$ is of rank one and $Sy_0^*\in \R^+Ty_0^*$.
\item If $T \in \Lin (X,Y)$ satisfies that $\|T\| = \|T^* (y_0^*)\|$ and $T^* (y_0^*) \in \SE(X)$ for some $y_0^* \in B_{Y^*}$, then $T$ attains its norm at a strongly exposed point. Hence, by (2), $T \in \overline{\ASE(X,Y)}$.
\end{enumerate}
\end{lemma}

Some comments on the previous results may be of interest.

\begin{remark}
\begin{enumerate}
\setlength\itemsep{0.3em}
\item
The facts that an operator $T \in \Lin (X,Y)$ attains its norm at $x_0\in \strexp{B_{X}}$ and that $y_0^*\in S_{Y^*}$ satisfies that $|y_0^* (T(x_0))| = \|T\|$ do not imply $T^* (y_0^*) \in \SE(X)$. For instance, take $x_0 = (1,1) \in \ell_\infty^2$ and $x_0^* =(1,0) \in \ell_1^2\equiv (\ell_\infty^2)^*$. Then $T := x_0^* \otimes x_0 \in \Lin (\ell_\infty^2 , \ell_\infty^2)$ attains its norm at $x_0\in \strexp{B_{\ell_\infty^2}}$, $x_0^*(T(x_0))=\|T\|$, but  $T^* (x_0^*) = x_0^*$ is not an exposing functional.
\item
Even if $T^* y_0^* \in \SE(X)$ and $\|T^* y_0^* \| = \|T^*\|$ for some $y_0^* \in S_{Y^*}$, $T$ may be not absolutely strongly exposing. For this, take the identity operator $\id$ on $\ell_2$; then $\id^* y^* \in \SE(\ell_2)$ for every $y^* \in \ell_2^*$, but $\id$ is not in $\ASE(\ell_2,\ell_2)$.
\end{enumerate}
\end{remark}

Related to item (4) of Lemma~\ref{lem:RMI} is the following easy fact which will be used all along the paper.

\begin{fact}\label{fact:containinginNA(X,Y)}
Let $X$, $Y$ be Banach spaces and $T\in \Lin(X,Y)$. Then, $T\in \NA(X,Y)$ if and only if $T^*\in \NA(Y^*,X^*)$ and there is $y^*\in S_{Y^*}$ such that $\|T^*y^*\|=\|T^*\|$ with $T^*y^*\in \NA(X,\K)$. In this case, $T$ attains its norm at the points where $T^*y^*$ does.
\end{fact}

A first consequence of Lemma~\ref{lem:RMI} is that the denseness of $\SE(X)$ is necessary to have denseness of absolutely strongly exposing operators.

\begin{proposition}\label{prop:goingdown}
Let $X$ be a Banach space. If $\ASE(X,Y)$ is dense in $\Lin (X,Y)$ for some nontrivial space $Y$, then $\SE(X)$ is dense in $X^*$.
\end{proposition}

\begin{proof}
Let $x^* \in S_{X^*}$ and $\eps >0$ be given. Fix $y_0 \in S_Y$ and consider $T = x^* \otimes y_0 \in \Lin (X,Y)$. By assumption, there is $S \in \ASE(X,Y)$ such that $\|S\|=1$ and $\|S-T\| < \eps$. Let say $\|S \| = \|S x_0\|$ for some $x_0 \in \strexp{B_{X}}$ and take $y^* \in S_{Y^*}$ so that $y^* (S(x_0)) = 1$ (hence $S^*y^*\in \SE(X)$ by Lemma~\ref{lem:RMI}).
Note that $|y_0^* (S(x_0)-T(x_0))| < \eps$, so $|y^* (y_0)||x^*(x_0)| > 1 -\eps$. In particular, $|y^* (y_0)| > 1-\eps$. Pick $\theta\in \T$ such that $\theta y^*(y_0)=|y^* (y_0)|$. We observe that
\[
\|\theta S^* y^* - x^*\| \leq \|\theta S^* y^* -\theta y^*(y_0) x^*\| + \|\theta y^* (y_0) x^* -x^*\| < 2\eps.
\]
As $S^* y^* \in \SE(X)$, the same happens with $\theta S^* y^*\in \SE(X)$, finishing the proof.
\end{proof}

The next characterization taken from \cite{H1975} relates differentiability points of $\Lin(X,Y)$ with absolutely strongly operators.

\begin{proposition}[\mbox{\cite[Theorem 3.1]{H1975}}]\label{prop:Heinrich}
Let $X$, $Y$ be Banach spaces and $T\in \Lin(X,Y)$. Then, the norm of $\Lin(X,Y)$ is Fr\'{e}chet differentiable at $T$ if and only if $T$ absolutely strongly exposes a point $x_0\in S_X$ and $Tx_0$ is a point of Fr\'{e}chet differentiability of $Y$.
\end{proposition}

Observe that, in particular, the existence of Fr\'{e}chet differentiability points of the norm of $\Lin(X,Y)$ implies the existence of Fr\'{e}chet differentiability points of the norm of $X^*$ and of the norm of $Y$.

\subsection{Some consequences of the residuality of norm attaining operators}\label{subsect:residuality-easy-consequences}

Our aim in this subsection is to show some implications of the residuality of the set of norm attaining operators. The next result contains the particularization to the case of operators of some folklore results on residual sets on Banach spaces.

\begin{proposition}\label{prop:affine}
Let $X$ and $Y$ be Banach spaces and suppose that $\NA (X,Y)$ is residual.
\begin{enumerate}
\setlength\itemsep{0.3em}
\item[(a)] Given $S \in \Lin (X,Y)$, the set $\mathcal A (S):= \{T \in \Lin (X,Y)\colon S +T \in \NA (X,Y)\}$ is residual.
\item[(b)] Given a sequence $\{S_n\}$ in $\Lin (X,Y)$ and $\eps >0$, there exists $T \in \Lin (X,Y)$ with $\|T\| < \eps$ such that $T+S_n \in \NA(X,Y)$ for every $n \in \N$.
\item[(c)] $\Lin(X,Y)=\NA(X,Y)-\NA(X,Y)$.
\end{enumerate}
\end{proposition}

\begin{proof}
(a): It is an immediate consequence of the fact that homeomorphisms conserve $G_\delta$-dense sets applied to the affine map
$\Phi_S\colon\Lin (X,Y) \longrightarrow \Lin (X,Y)$ given by $\Phi_S (T) = T-S$ for every $T \in \Lin (X,Y)$, as $\Phi_S(\NA(X,Y))=\mathcal{A}(S)$.

(b): By (a), the set $\mathcal{A}(S_n)$ is residual for each $n \in \N$; hence $\bigcap_n \mathcal{A}(S_n)$ is residual. In particular, there exists $T \in \Lin (X,Y)$ with $\|T \| < \eps$ such that $T \in \bigcap_n \mathcal{A}(S_n)$, meaning that $T+S_n\in \NA(X,Y)$ for every $n\in \N$.

(c): For a given $S\in \Lin(X,Y)$, $\mathcal{A}(S)\cap \NA(X,Y)$ is non-empty, so there is $T\in \NA(X,Y)$ such that $S+T\in \NA(X,Y)$.
\end{proof}

Observe that, in the proof given for Proposition~\ref{prop:affine}, we do not use any special property of $\NA(X,Y)$ more than its residuality, so it can be also written in terms of $\ASE(X,Y)$.

\begin{proposition}\label{prop_affineASE}
Let $X$ and $Y$ be Banach spaces and suppose that $\ASE(X,Y)$ is dense.
\begin{enumerate}
\setlength\itemsep{0.3em}
\item[(a)] Given $S \in \Lin (X,Y)$, the set $\mathcal A (S):= \{T \in \Lin (X,Y)\colon S +T \in \ASE(X,Y)\}$ is residual.
\item[(b)] Given a sequence $\{S_n\}$ in $\Lin (X,Y)$ and $\eps >0$, there exists $T \in \ASE(X,Y)$ with $\|T\| < \eps$ such that $T+S_n \in \ASE(X,Y)$ for every $n \in \N$.
\item[(c)] $\Lin(X,Y)=\ASE(X,Y)-\ASE(X,Y)$.
\end{enumerate}
\end{proposition}

It is easy to give examples showing that the residuality of $\NA(X,Y)$ is necessary in Proposition~\ref{prop:affine} (or the denseness of $\ASE(X,Y)$ in Proposition~\ref{prop_affineASE}).

\begin{example}
Let $X=c_0$. Then, $\NA(X,\K)= \ell_1 \cap c_{00} \subseteq \ell_1$, so it is not residual. Besides, $\ASE(X,\K)=\SE(X)=\{0\}$ since the norm of $\ell_1$ is nowhere Fr\'{e}chet differentiable. Moreover:
\begin{itemize}
\setlength\itemsep{0.3em}
  \item $\NA(X,\K)-\NA(X,\K)= \ell_1 \cap c_{00}\neq  \ell_1 $; $\ASE(X,\K)-\ASE(X,\K)=\{0\}$.
  \item Given $x_1^* =0$ and $x_2^* \in \ell_1 \setminus c_{00}$, there is no $x^* \in \NA(X,\K)$ such that $x_1^* + x^* \in \NA(X,\K)$ and $x_2^*+x^*\in \NA(X,\K)$.
\end{itemize}
\end{example}

\section{Necessary conditions for the Bishop-Phelps property and for property A}\label{sec:BPandpropertyA}
Our main result here is the following.

\begin{theorem}\label{thm:LURrenorming}
Let $X$ be a Banach space and let $C$ be a bounded subset of $X$ with the Bishop-Phelps property.
\begin{itemize}
\setlength\itemsep{0.3em}
  \item[(a)] If $X$ admits an equivalent LUR renorming, then $\SE(C)$ is dense in $X^*$. In particular, $C$ is contained in the closed convex hull of its strongly exposed points.
  \item[(b)] If $X$ admits an equivalent strictly convex norm, then the set of exposing functionals of $C$ is dense in $X^*$. In particular, $C$ is contained in the closed convex hull of its exposed points.
\end{itemize}
\end{theorem}

We need a preliminary lemma to prove the theorem. Recall that a \emph{monomorphism} between two Banach spaces $X$ and $Y$ is an operator $T\in \Lin(X,Y)$ which is an isomorphism from $X$ onto $T(X)$. It is well known that $T\in \Lin(X,Y)$ is a monomorphism if and only if there is $C>0$ such that $\|Tx\|\geq C\|x\|$ for all $x\in X$, and if and only if $\ker T=\{0\}$ and $T(X)$ is closed (see \cite[\S~10.2.3]{Kadets_book}, for instance). It is also a classical result that the set of monomorphisms between Banach spaces is open (see \cite[Lemma~2.4]{Abramovich-Aliprantis}, for instance).

\begin{lemma}\label{lemma:C-monomorphismattainingsupremum}
Let $X$, $Y$ be Banach spaces, $T\in \Lin(X,Y)$ be a monomorphism, and let $C\subset X$ be bounded.
\begin{itemize}
\setlength\itemsep{0.3em}
  \item[(1)] If $x_0\in C$ satisfies that $\|T x_0\|=\sup\{\|Tx\|\colon x\in C\}$ and that $Tx_0$ is an LUR point of $Y$, then $x_0$ is a strongly exposed point of $C$. Moreover, $x_0$ is strongly exposed by $T^*y^*$ for every $y^*\in S_{Y^*}$ such that $\re y^*(Tx_0)=\sup\{\|Tx\|\colon x\in C\}$.
  \item[(2)] If $x_0\in C$ satisfies that $\|T x_0\|=\sup\{\|Tx\|\colon x\in C\}$ and that $Tx_0$ is a rotund point of $Y$, then $x_0$ is a exposed point of $C$. Moreover, $x_0$ is exposed by $T^*y^*$ for every $y^*\in S_{Y^*}$ such that $\re y^*(Tx_0)=\sup\{\|Tx\|\colon x\in C\}$.
\end{itemize}
\end{lemma}

\begin{proof}
The proof of both assertions is almost the same, so we only provide the one of (1), the one in which we are more interested. Take $y^*\in S_{Y^*}$ such that
$$
\re y^*(Tx_0)=\|Tx_0\|=\sup\{\|Tx\|\colon x\in C\}.
$$
First, observe that
\begin{align*}
 \re T^*y^*(x_0) &=\re y^*(Tx_0)=\|T x_0\|=\sup\{\|Tx\|\colon x\in C\}\\ &\geq \sup\{\re y^*(Tx)\colon x\in C\}= \sup\{\re T^*y^*(x)\colon x\in C\}.
\end{align*}
Moreover, if a sequence $\{x_n\}\subset C$ satisfies that
$$
\lim\limits_n\re T^* y^*(x_n)=\sup\{\re T^*y^*(x)\colon x\in C\}=\|Tx_0\|,
$$
we have that $\|Tx_n\|\leq \|Tx_0\|$ and
$$
\lim\limits_n\|Tx_n + Tx_0\|\geq \lim\limits_n \re T^*y^*(x_n+x_0) =2\|Tx_0\|.
$$
As $Tx_0$ is an LUR point, this implies that $\|Tx_n-T x_0\|\longrightarrow 0$. But now $T$ is a monomorphism, so it is bounded from below, which implies that $\lim\limits_n x_n =x_0$. In other words, $T^*y^*$ strongly exposes $C$ at $x_0$ and, in particular, $x_0$ is a strongly exposed point of $C$.
\end{proof}

We are ready to present the pending proof.

\begin{proof}[Proof of Theorem \ref{thm:LURrenorming}]
We only include the arguments to get item (a). The proof of item (b) follows the same lines using item (2) of Lemma~\ref{lemma:C-monomorphismattainingsupremum} instead of item (1).

We write $\| \cdot \|$ for the given norm of $X$. Consider a norm $\nn{ \cdot}$ on $X$ which is LUR and satisfies $\nn{x} \leq \|x\|$ for every $x \in X$. Define $Y:= (X, \nn{\cdot}) \oplus_2 \mathbb{K}$ and note that $Y$ is LUR. Pick $x^* \in X^*\setminus \{0\}$ and $\eps>0$. As $\SE(C)=\SE(C + x_0)$ for every $x_0\in X$, without loss of generality, we may assume that
\begin{equation}\label{eq:proofSEC-x*positive}
x^*(C)\subset \bigl\{r\e^{i\theta}\colon r\geq 1,\, |\theta|\|x^*\|\leq\eps/2 \bigr\}
\end{equation}
(in the real case this is just $x^*(C)\subset [1,+\infty[$).
For each $n \in \N$, define $T_n \in \Lin (X,Y)$ by $T_n x = (n^{-1} x, x^* (x))$ for every $x\in X$. Observe that each $T_n$ is a monomorphism.

Define $S \in \Lin (X,Y)$ by $S x = (0, x^* (x))$ for every $x \in X$ and observe that $\|T_n - S \| \longrightarrow 0$. Since the set of monomorphisms from $X$ to $Y$ is open and $C$ has the Bishop-Phelps property, we may find a sequence $\{S_n\}$ of monomorphisms from $X$ to $Y$ which attain the supremum of their norms on $C$ and $\lim\limits_n\|T_n-S_n\|=0$. Therefore, $\|S_n - S \| \longrightarrow 0$. As every $S_n$ is a monomorphism attaining the supremum of its norms on $C$ and $Y$ is LUR, item (1) of Lemma~\ref{lemma:C-monomorphismattainingsupremum} provides a sequence $\{x_n\}$ of points of $C$ and a sequence $\{y_n^*\}$ of elements of $S_{Y^*}$ such that each $S_n^*y_n^*$ belongs to $\SE(C)$ and strongly exposes $C$ at $x_n$. We write $y_n^* = (z_n^*, \lambda_n)\in Y^* = X^*\oplus_2 \mathbb{K}$ and, passing to a subsequence, assume that $\lambda_n \longrightarrow \lambda_0$ for some $\lambda_0 \in \mathbb{K}$. Since $\|S^* y_n^* - S_n^* y_n^*\| \longrightarrow 0$ and $S^*y_n^*=\lambda_n x^*$ for every $n\in \N$, we have that
\begin{equation}\label{eq:proofSEC-final}
\| \lambda_0 x^* - S_n^* y_n^*\| \longrightarrow 0.
\end{equation}
Now, we set
$$
\alpha_n:=y_n^*(S_n(x_n))=\sup\{\|S_n(x)\|\colon x\in C\}
$$
and observe that
$$
|\lambda_n x^*(x_n)-\alpha_n| = \bigl|y_n^*(Sx_n)-y_n^*(S_n x_n)\bigr|\leq
\|S - S_n\|\sup_n\|x_n\|\longrightarrow 0.
$$
Passing to a subsequence, we may suppose that $\alpha:=\lim\limits_n \alpha_n$ and $\beta:=\lim\limits_n x^*(x_n)$ exist, and we obtain from \eqref{eq:proofSEC-final} that $\lambda_0 \beta = \alpha$.

Notice from \eqref{eq:proofSEC-x*positive} that $\sup\{\|S(x)\|\colon x\in C\} = \sup \{ |x^* (x)|\colon x \in C \} \geq 1$; hence we get that $\alpha \geq 1$ since $\|S_n - S \| \longrightarrow 0$. In particular, $\lambda_0\neq 0$. Besides, using again \eqref{eq:proofSEC-x*positive}, we may write $\beta:=r\e^{i \theta}$ with $r\geq 1$ and $|\theta|\|x^*\|\leq\eps/2$. Now,
$$
\lambda_0=\alpha \beta^{-1}=|\lambda_0|\e^{-i\theta},
$$
and so
\begin{align*}
\left\|x^*-\frac{\lambda_0}{|\lambda_0|} x^* \right\| &=\bigl|1-\e^{-i\theta}\bigr|\|x^*\|= 2\bigl|\sin(\theta/2)\bigr|\|x^*\|\leq \eps/2.
\end{align*}
From \eqref{eq:proofSEC-final}, and since $\lambda_0\neq 0$, we have that			
\begin{equation*}
\Bigl\|\lambda_0|\lambda_0|^{-1} x^* - |\lambda_0|^{-1} S_n^* y_n^*\Bigr\| \longrightarrow 0,
\end{equation*}
so we may find $n\in \N$ such that $\bigl\|x^* - |\lambda_0|^{-1} S_n^* y_n^*\bigr\|<\eps$.
The arbitrariness of $\eps>0$ and the fact that $\lambda\SE(C)=\SE(C)$ for every $\lambda\in \R^+$ finish the proof.
\end{proof}

Some remarks on the previous result are pertinent.

\begin{remark}
If the set $C$ in Theorem~\ref{thm:LURrenorming} is balanced, then the proof slightly simplifies. Indeed, in this case we have that $\lambda\SE(C)=\SE(C)$ for every $\lambda\in \K\setminus\{0\}$ and so we only need to prove that $\lambda_0\neq 0$, a easier fact to show.
\end{remark}

\begin{remark}
Observe that we do not need convexity nor closedness of the set $C$ in Theorem~\ref{thm:LURrenorming}.
      \begin{enumerate}
        \item With respect to convexity, this is not very important as the set of strongly exposing functionals of a set and the one of its convex hull coincide and, on the other hand, a set has the Bishop-Phelps property if and only if it convex hull does.
        \item With respect to closedness, the situation is different. On the one hand, $\SE(C)$ and $\SE(\overline{C})$ may be completely different, and it is not true that $C$ has the Bishop-Phelps property whenever $\overline{C}$ does (while the other implication is clear).
        \item Let us also comment here that the Bishop-Phelps property of $C$ does not imply $C$ to be closed: just consider a square in the plane for which we have removed the sides but not the vertices.
      \end{enumerate}
\end{remark}

\begin{remark} Theorem~\ref{thm:LURrenorming} improves results of Lindenstrauss \cite[Theorem 2]{Lin}, where $C$ is the unit ball and only the fact that $B_X$ is the closed convex hull of the strongly exposed points is obtained. Besides, the fact that $\SE(C)$ is dense in $X^*$ was previously known for weakly compact convex sets (Bourgain \cite{Bou2} and Lau \cite{Lau}) and for bounded closed convex sets with the RNP (Bourgain \cite{Bou}).
\end{remark}

Applying Theorem~\ref{thm:LURrenorming} to the unit ball of a Banach space, we get the following improvement of the necessary conditions given by Lindenstrauss in \cite[Theorem 2]{Lin}.

\begin{corollary}\label{corollary:PropertyAimpliesSEdense}
Let $X$ be a Banach space with property A.
\begin{enumerate}
\setlength\itemsep{0.3em}
\item[(a)] If $X$ admits an LUR renorming, then $\SE(X)$ is dense in $X^*$.
\item[(b)] If $X$ admits an strictly convex renorming, then functionals exposing $B_X$ are dense in $X^*$.
\end{enumerate}
\end{corollary}

We will present in Example~\ref{example:Lindenstrauss-not-sufficient} a separable Banach space $X$ such that $B_X$ is the closed convex hull of its strongly exposed points but $\SE(X)$ is not dense in $X^*$ (even more, exposing functionals are not dense in $X^*$), see also Remark~\ref{remark:torusnotA}. In particular, this space fails property A while it fulfills the necessary condition provided by Lindenstrauss in \cite[Theorem 2]{Lin}, so Corollary~\ref{corollary:PropertyAimpliesSEdense} really improves Lindenstrauss result. As far as we know, no example of such phenomenon has already appeared in the literature (that is, an example of a Banach space $X$ for which the unit ball is the closed convex hull of its strongly exposed points but $\SE(X)$ is not dense in $X^*$). From the isomorphic point of view, it is known that a separable Banach space $X$ has the RNP if and only if every equivalent renorming of $X$ satisfies one (and so all) of the following properties (see \cite[Theorem~3.4]{GMZ} for instance): (i) the unit ball contains slices of arbitrary small diameter, (ii) the unit ball is the closed convex hull of its strongly exposed points, (iii) the strongly exposing functionals are dense in $X^*$. It is immediate that conditions (i) and (ii) are not equivalent for a concrete norm (containing just one strongly exposed point implies dentability). Remark~\ref{remark:torusnotA} shows that conditions (ii) and (iii) are neither equivalent for a concrete norm.

\section{Sufficient conditions for the denseness of $\ASE(X,Y)$}\label{section:sufficient}

Our aim here is to provide conditions on a Banach space $Y$ ensuring that $\ASE(X,Y)$ is dense, provided $\SE(X)$ is dense.

\subsection{When the range space satisfies some previously known conditions}\label{subsection31-known}

We start showing that the known conditions for a Banach space $Y$ to have Lindenstrauss property B actually imply $\ASE(X,Y)$ to be dense when $\SE(X)$ is dense in $X^*$. As far as we know, there are only two properties studied in the literature which imply Lindenstrauss property B: property $\beta$ introduced by Lindenstrauss himself in the seminal paper \cite{Lin} and the weaker property quasi-$\beta$ introduced by Acosta, Aguirre, and Pay\'{a} in 1996 \cite{AAP2}.
A Banach space $Y$ is said to have {\it property quasi-$\beta$} if there exist $A = \{y_{\lambda}^*\colon \lambda \in \Lambda\} \subseteq S_{Y^*}$, a mapping $\sigma\colon A \longrightarrow S_Y$, and a function $\rho\colon A \longrightarrow \R$ satisfying
	\begin{itemize}
		\setlength\itemsep{0.3em}
		\item[(i)] $y_{\lambda}^*(\sigma(y_{\lambda}^*)) = 1$ for every $\lambda\in\Lambda$,
		\item[(ii)] $|z^* (\sigma(y^*))| \leq \rho (y^*) < 1$ whenever $y^*, z^* \in A$ with $y^* \neq z^*$,
		\item[(iii)] for every $e^* \in \ext{B_{Y^*}}$, there exists a subset $A_{e^*} \subseteq A$ and $t \in \mathbb{C}$ with $|t|=1$ such that $te^* \in \overline{A_{e^*}}^{w^*}$ and $\sup \{ \rho(y^*)\colon y^* \in A_{e^*}\} < 1$. 	
	\end{itemize}
If there is $0 \leq R<1$ such that $\rho(y^*)\leq R$ for all $y^*\in A$, then the space $Y$ has \emph{property $\beta$} introduced by Lindenstrauss (with an equivalent formulation). Examples of Banach spaces with property $\beta$ are finite-dimensional spaces whose unit ball is a polytope (in the complex case, those spaces for which the set of extreme points of the dual ball is finite up to rotation) and closed subspaces of $\ell_\infty$ containing the canonical copy of $c_0$. There are examples of Banach spaces with property quasi-$\beta$ which do not have property $\beta$ (and we will show some more in Remark~\ref{rem:examqbetanotbeta}), including some finite-dimensional real spaces whose dual unit ball has infinitely many extreme points and the so-called Gowers space \cite[Example~7]{AAP2} (which is an isometric predual of the Lorentz sequence space $d(\{1/n\},1)$). We refer the interested reader to \cite{AAP2}.

The following result can be proved in the same way as in \cite[Theorem 2]{AAP2} but using the denseness of $\SE(X)$ instead of the Bishop-Phelps theorem.

\begin{theorem}\label{thm:quasibeta}
Let $X$, $Y$ be Banach spaces. Suppose that $\SE(X)$ is dense in $X^*$ and that $Y$ has property quasi-$\beta$. Then, for every closed subspace $\mathcal{I}(X,Y)$ of $\Lin(X,Y)$ containing rank one operators, $\ASE (X,Y)\cap \mathcal{I}(X,Y)$ is dense in $\mathcal{I}(X,Y)$. In particular, $\ASE(X,Y)$ is dense in $\Lin(X,Y)$ and $\ASE(X,Y)\cap \cpt(X,Y)$ is dense in $\cpt(X,Y)$
\end{theorem}

\begin{proof}
Let $T \in \Lin (X,Y)$, $\|T\|=1$ and $\eps >0$ be given. Due to a result by Zizler \cite[Proposition~4]{Z}, there is $S_1\in \Lin(X,Y)$ with $\|S_1\|=1$ such that $\|T-S_1\|<\eps$ and $S_1\in \NA(Y^*,X^*)$. Going into the proof of \cite[Proposition~4]{Z}, one realizes that when $T\in \mathcal{I}(X,Y)$, then $S_1\in \mathcal{I}(X,Y)$ as $T-S_1$ is the limit of a sequence of operators of finite rank. On the other hand, by a result of Johannesen (see \cite[Theorem~5.8]{Lima1978}), $S_1^*$ attains its norm at an extreme point $e^*$ of $B_{Y^*}$. As $Y$ has property quasi-$\beta$, there exists $A_{e^*} \subseteq A$ and $t \in \mathbb{T}$ such that $te^* \in \overline{A_{e^*}}^{w^*}$ and $\eta := \sup \{ \rho(y^*)\colon y^* \in A_{e^*} \} < 1$. Fix $\gamma >0$ so that
\[
1+\eta \left( \frac{\eps}{2} + \gamma \right) < \left(1+\frac{\eps}{2}\right)(1-\gamma)
\]
and find $y_1^* \in A_{e^*}$ such that $\|S_1^* y_1^* \| > 1 - \gamma$. Since $\SE(X)$ is dense, there exists $z^* \in \SE(X)$ such that $\| z^* - S_1^* y_1^* \| < \gamma$ and $\|z^* \| = \|S_1^* (y_1^*)\|$. Define $S_2 \in \mathcal{I}(X,Y)$ by
\[
S_2(x) = S_1(x) + \left[ \left(1+\frac{\eps}{2} \right) z^* (x) - S_1^* (y_1^*) (x)\right]y_1
\]
for every $x \in X$, where $y_1 = \sigma(y_1^*)$. Arguing as in the proof of \cite[Theorem 2]{AAP2}, we have
\begin{enumerate}
\setlength\itemsep{0.3em}
\item $\| S_2-S_1\| < \eps$,
\item $S_2^* (y_1^*) =\left(1+\frac{\eps}{2}\right) z^*$,
\item $\|S_2^*\| = \|S_2^*(y_1^*)\|$.
\end{enumerate}
Since $z^*\in \SE(X)$, it follows from (2), (3), and Lemma~\ref{lem:RMI}, that there exists $S_3 \in \ASE(X,Y)\cap \mathcal{I}(X,Y)$ such that $\|S_2- S_3\| < \eps$; hence $\|T- S_3\|<3\eps$. This completes the proof.
\end{proof}

Let us present now new examples of Banach spaces with property quasi-$\beta$. We need some notation. Given a Banach space $Y$, let us consider the equivalence relation on $\ext{B_{Y^*}}$ given by $x^* \sim y^*$ if and only if $x^* = \lambda y^*$ for some $\lambda \in \mathbb{T}$. We write $E_Y$ to denote the topological space $\ext{B_{Y^*}}/\sim$ endowed with the quotient topology of the weak star topology. A \emph{real} Banach space is said to be \emph{polyhedral} if the unit ball of any of its finite-dimensional subspaces is a polyhedron (the convex hull of finitely many points). We refer to \cite{FV2004} and references therein for an exhaustive account on different definitions of polyhedrality.

\begin{example}\label{Example:quasibeta_example}
The following are examples of Banach spaces having property quasi-$\beta$
\begin{enumerate}
\setlength\itemsep{0.3em}
\item Preduals of $\ell_1$ which are polyhedral (real case).
\item Banach spaces $Y$ for which $E_Y$ is discrete (i.e.\ it has no accumulation points). In particular:
\begin{enumerate}
  \item a real Banach space $Y$ which satisfies that the $w^*$-accumulation points of $\ext{B_{Y^*}}$ belongs to the norm interior of $B_{Y^*}$ (they are the so-called (III)-polyhedral spaces following \cite[Definition 1.1]{FV2004});
  \item arbitrary closed subspaces of (real or complex) $c_0(\Gamma)$.
\end{enumerate}
\end{enumerate}
\end{example}

As far as we know, assertion (1) was unknown. Assertion (2) appeared in the PhD dissertation of F.~Aguirre (see \cite[Teorema 1.20]{AguirrePhD}) but it has not been published in the journal literature. Its consequences for (III)-polyhedral spaces and for closed subspaces of $c_0(\Gamma)$, while easy, seem to be new.

\begin{proof}
(1): Suppose that $Y$ is a \emph{real} polyhedral $\ell_1$-predual space. Notice that every extreme point of $B_{Y^*}$ is $w^*$-exposed \cite[Lemma 3.3]{Sp}. So, for $y^* \in \ext{B_{Y^*}}$, we can consider $\sigma(y^*) \in S_Y$ such that $y^* (\sigma(y^*)) = 1$ and $|z^* (\sigma(y^*))| < 1$ whenever $z^* \not\in \{y^*, -y^*\}$.
Next, by \cite[Theorem~4.1]{CMPV2016} we have that
\[
\rho(y^*) := \sup \bigl\{ |z^* (\sigma(y^*))|\colon z^* \in \ext{B_{Y^*}} \setminus \{ \pm y^* \} \bigr\} < 1
\]
for every $y^* \in \ext{B_{Y^*}}$ (this is called (BD) polyhedrality in \cite{CMPV2016}).
Now, consider the set $A \subseteq S_{Y^*}$ given by $A = \{y^* \in \ext{B_{Y^*}}\colon \mathfrak{u} (y^*) =1 \}$, where $\mathfrak{u}$ is the vector in $Y^{**}$ which corresponds isometrically to $(1,1,\ldots) \in \ell_\infty$. Since $\mathfrak{u}$ is an extreme point of $B_{Y^{**}}$, we have that $|\mathfrak{u} (y^*)|=1$ for every $y^* \in \ext{B_{Y^*}}$ (see \cite[Corollary 2.8]{KMMP}, for instance). For each $e^* \in \ext{B_{Y^*}}$, take $t \in \{1,-1\}$ so that $\mathfrak{u}(te^*) = 1$ and set $A_{e^*} := \{ t e^*\}\subset A$. Therefore, the set $A$ and the mappings $\sigma$ and $\rho$ satisfy the conditions (i)-(iii) of property quasi-$\beta$.

(2): Suppose that $Y$ is a real or complex Banach space satisfying that $E_Y$ contains no accumulation points. By \cite[Proposition 2.2]{BM2006}, this implies that every point in $\ext{B_{Y^*}}$ is $w^*$-strongly exposed. The Axiom of Choice allows us to consider a subset $A$ of $\ext{B_{Y^*}}$ which is consists of a unique representative of each equivalence class. For $y^* \in A$, let $\sigma(y^*)$ be an element in $S_Y$ which strongly exposes $y^*$. Observe that
\[
\rho (y^*) := \sup\{ |z^* (\sigma(y^*))|\colon z^* \in A, z^* \neq y^* \} < 1
\]
for each $y^* \in A$. Indeed, otherwise, we may find a sequence $\{z_n^*\} \subseteq A$ with $z_n^* \neq y^*$ (so $[z_n^*]\neq [y^*]$ by the way we have selected $A$) such that $z_n^* (\sigma(y^*)) \longrightarrow 1$. Since $\sigma(y^*)$ strongly exposes $y^*$, we get that $\{z_n^*\}$ converges in norm to $y^*$. This implies that the sequence of equivalence classes $\{[z_n^*]\}$ converges to the equivalence class $[y^*]$, which contradicts the fact that $E_Y$ has no accumulation points. Finally, given $e^* \in \ext{B_{Y^*}}$, let $t \in \mathbb{C}$ with $|t|=1$ such that $\overline{t} e^* \in A$ and set $A_{e^*} := \{ \overline{t} e^*\}$. Then $e^* \in t A_{e^*}$ and $\sup \{ \rho(y^*)\colon y^* \in A_{e^*}\} < 1$. This shows that $Y$ has property quasi-$\beta$.

Finally, it is immediate that (a) implies that $E_Y$ is discrete. To get (b), it is immediate that $Y = c_0(\Gamma)$ satisfies that $0$ is the unique $w^*$-accumulation point of $\ext{B_{Y^*}}$, and this property clearly goes down to closed subspaces (see, for instance, \cite[Theorem 1.2]{FV2004}).
\end{proof}

\begin{remark}\label{rem:examqbetanotbeta}
\begin{enumerate}
\setlength\itemsep{0.3em}
\item Observe that \cite[Theorem 2]{AAP2} and Example \ref{Example:quasibeta_example} show that closed subspaces of $c_0$ have property B. As far as we know, this result is new.
\item Also, by the proof of \cite[Theorem 2]{AAP2}, it follows from Example \ref{Example:quasibeta_example} that for every closed subspace $Y$ of $c_0$, $\NA(X,Y)\cap \cpt(X,Y)$ is dense in $\cpt(X,Y)$ for every Banach space $X$. As far as we know, this result is also new. It was known with the extra assumption that $Y$ has the approximation property (and in this case every element in $\cpt(X,Y)$ can be approximated by elements in $\NA(X,Y)$ of finite rank), see \cite[Example~4.7]{martinjfa}.
 \item There are closed subspaces of $c_0$ without property $\beta$. Indeed, for each $k \in \N$, consider $Y_k = \mathbb{R}^2$ endowed with the norm $\|(x,y)\| = \max \{ |x|, |y| + \frac{1}{k}|x| \}$. Viewing $Y_k$ as a closed subspace of the three dimensional $\ell_\infty$ space, the space $Y:= [\oplus_{k=1}^\infty Y_k]_{c_0}$ is a closed subspace of $c_0$. It is known that $Y$ lacks property $\beta$ (see the arguments in \cite[Example~4.1]{ACKLM2015}).
\end{enumerate}
\end{remark}

Dealing with compact operators, there are more sufficient conditions on a Banach space $Y$ than the property quasi-$\beta$ to ensure that $\NA(X,Y)\cap \cpt(X,Y)$ is dense in $\cpt(X,Y)$ for every Banach space $X$. We refer to \cite{Martin2016} for a detailed account. Some of the results have a counterpart for $\ASE(X,Y)\cap \cpt(X,Y)$ when $\SE(X)$ is dense. The following is one of interesting examples.

\begin{example}\label{thm:cpt}
Let $X$ be a Banach space such that $\SE (X)$ is dense in $X^*$ and let $Y$ be a Banach space such that $Y^*\equiv L_1(\mu)$ for some measure $\mu$. Then $\ASE (X,Y) \cap \cpt (X,Y) $ is dense in $\cpt (X,Y)$.
\end{example}

The proof is motivated by the corresponding result of Johnson and Wolfe \cite{JW1979} for norm attaining compact operators.

\begin{proof}
Let $T \in \cpt(X, Y)$ and $\eps >0$ be given. Take $\{y_1,\ldots, y_n\}$ a $\frac{\eps}{8}$-net of $T(B_{X})$. By results of Lazar and Lindenstrauss in the real case (see \cite[Theorem~3.1]{LazarLindenstrauss}) and Nielsen and Olsen in the complex case (see \cite[Theorem~1.3]{NieOls}), we may find a $1$-complemented subspace $E$ of $Y$ such that $E$ is isometric to $\ell_\infty^m$ for some $m \in \N$ and for each $i=1,\ldots, n$, there exists $e_i \in E$ so that $\|y_i -e_i \|< \frac{\eps}{8}$. Let us denote by $P$ a norm one projection from $Y$ onto $E$ and write $J\colon E\longrightarrow Y$ for the canonical inclusion. For each $x \in B_X$, there exists $e \in E$ such that $\|T(x)-e\|<\frac{\eps}{4}$; hence
\[
\|T(x)-JPT(x)\| \leq \|T(x)-J(e)\|+\|e-PT(x)\| < \frac{\eps}{2}.
\]
This shows that $\|T - JPT\| \leq \frac{\eps}{2}$. Since $PT \in \cpt(X, E)$ and $E$ is isometric to $\ell_\infty^m$, by Theorem~\ref{thm:quasibeta}, there exists $G \in \ASE(X,E)$ such that$\|PT-G\| < \frac{\eps}{2}$, so
$$
\|T-JG\| \leq \|T-JPT\| + \|JPT-JG\|\leq \|T-JPT\|+\|PT-G\|<\eps.
$$
Finally, $JG\in \ASE(X,Y)\cap\cpt(X,Y)$.
\end{proof}

Beside the property of being the predual space of $L_1$-space, there is another property, called \emph{ACK$_\rho$ structure}, on the range space $Y$ which guarantees that $\ASE (X, Y) \cap \mathcal{K} (X,Y)$ is dense in $\mathcal{K} (X,Y)$ for every Banach space $X$ provided that $\SE (X)$ is dense in $X^*$. In order to establish the result, we need the following notation and definition. Recall from \cite{CGKS2018} that a Banach space $Y$ is said to have \emph{ACK$_\rho$ structure} whenever there exists a $1$-norming set $\Gamma \subseteq B_{Y^*}$ such that for every $\eps >0$ and every nonempty relatively $w^*$-open subset $U \subseteq \Gamma$, there exist a nonempty subset $V \subseteq U$, $y_1^* \in V$, $e \in S_Y$ and an operator $T \in \Lin (Y,Y)$ with the following properties:
\begin{enumerate}
\setlength\itemsep{0.3em}
\item $\| Fe \| = \| F \| =1$,
\item $y_1^* (Fe) = 1$,
\item $F^* y_1^* = y_1^*$,
\item denoting $V_1 = \{ y^* \in \Gamma \colon \|F^* y^* \| + (1-\eps) \| (\text{Id}_{Y^*} - F^*) (y^*) \| \leq 1\}$, then $|y^* (Fe) | \leq \rho$ for every $y^* \in \Gamma \setminus V_1$,
\item $\dist (F^* y^*, \aconv\{0,V\}) < \eps_2$ for every $y^* \in \Gamma$,
\item $|v^* (e) - 1 | \leq \eps$ for every $v^* \in V$.
\end{enumerate}
Given Banach spaces $X$ and $Y$, and $\Gamma \subset Y^*$, an operator $T \in \Lin (X,Y)$ is said to be \emph{$\Gamma$-flat} \cite{CGKS2018} if $T^* \vert_\Gamma \colon (\Gamma, w^*) \longrightarrow (X^*, \| \cdot\|_{X^*})$ is openly fragmented. We denote the set of all $\Gamma$-flat operators by $\flat(X,Y)$. Among others results, it is known that every Asplund operator from $X$ to $Y$ is $\Gamma$-flat for every $\Gamma \subseteq Y^*$, and that $\Lin (X, Y) = \flat(X,Y)$ when $(\Gamma, w^*)$ is discrete.

We state the promised result which provides new information about the set $\ASE(X,Y)$ in presence of ACK$_\rho$ structure.

\begin{theorem}\label{propo:Gammaflat-just-denseness}
Let $X$ be a Banach space. If $\SE(X)$ is dense in $X^*$ and $Y$ is a Banach space having ACK$_\rho$ structure associated with a $1$-norming set $\Gamma \subseteq B_{Y^*}$, then $\flat(X,Y) \subseteq \overline{\ASE(X,Y)}$.
\end{theorem}

Before providing the proof of Theorem \ref{propo:Gammaflat-just-denseness}, let us derive some consequences about the density of $\ASE(X,Y)$.

Observe that if either $X$ or $Y$ is Asplund, then $\Lin (X, Y) = \flat(X,Y)$; hence we have the following.

\begin{corollary}\label{corollary:ACKrho-structure-XorY-Asplund}
Let $X$ and $Y$ be Banach spaces. If either $X$ or $Y$ is an Asplund space, $\SE(X)$ is dense in $X^*$ and $Y$ is a Banach space having ACK$_\rho$ structure, then $\ASE (X, Y)$ is dense in $\Lin (X, Y)$.
\end{corollary}

As compact operators are particular cases of Asplund operators, we also obtain the following consequence.

\begin{corollary}
Let $X$ be a Banach space such that $\SE(X)$ is dense in $X^*$ and let $Y$ be a Banach space having ACK$_\rho$ structure. Then, $\ASE (X, Y)\cap \cpt(X,Y)$ is dense in $\Lin (X, Y)\cap \cpt(X,Y)$.
\end{corollary}

\begin{remark}\label{rem:appliACKstructure}
Let us point out that the previous corollaries provide new examples of pairs $(X,Y)$ for which $\ASE(X,Y)$ (resp.\ $\ASE(X,Y)\cap \mathcal K(X,Y)$) is dense in $\Lin(X,Y)$ (resp.\ $\mathcal K(X,Y)$). Observe that we can require $X$ being Asplund and $\SE(X)$ being dense, and we only have to require on $Y$ having ACK$_\rho$ structure. Let us provide a list of examples of Banach spaces with ACK$_\rho$ structure (see \cite{CGKS2018} for details):
\begin{enumerate}
\setlength\itemsep{0.3em}
\item If $Y$ has property $\beta$, then $Y$ has ACK$_\rho$ structure.
\item If $K$ is a compact Hausdorff topological space and $Y$ has ACK$_\rho$ structure, then so does $C(K,Y)$.
\item A uniform algebra has ACK$_\rho$ structure.
\item The property of having ACK$_\rho$ structure is preserved by taking finite injective tensor products.
\item $c_0(Y)$ and $\ell_\infty(Y)$ has ACK$_\rho$ structure if $Y$ has ACK$_\rho$ structure.
\end{enumerate}
\end{remark}

Now it is time to prove Theorem \ref{propo:Gammaflat-just-denseness}. In order to do so, we will prove a stronger result, related with a version of Bishop-Phelps-Bollob\'{a}s result for absolutely strongly exposing operators, which has its own interest and from which Theorem \ref{propo:Gammaflat-just-denseness} will be obtained immediately.

To this end, we begin by introducing the following definition for functionals, which can be seen as a version of the  Bishop-Phelps-Bollob\'as theorem for $\SE(X)$.

\begin{definition}
A Banach space $X$ is said to have property [P] if there exists a function $\eps \in (0,1) \longmapsto \eta(\eps) >0$ such that whenever $\re x^* (x) > 1-\eta(\eps)$ for $x \in S_X$ and $x^* \in S_{X^*}$, then there exists $y^* \in \SE(X)$ and $y \in S_X$ such that $\|y^*\|=y^*(y)=1$, $\|y^* - x^*\| < \eps$ and $\|y-x\| <\eps$.
\end{definition}

Note that the property [P] implies not only that $\SE(X)$ is dense in $X^*$ but also that $\strexp{B_X}$ is dense in $B_X$. We write
$$
\Pi(X):=\bigl\{(x,x^*)\in S_X\times S_{X^*}\colon x^*(x)=1 \bigr\}.
$$

\begin{lemma}\label{lemma:charac-property-P}
Let $X$ be a Banach space. Then, the following assertions are equivalent:
\begin{enumerate}
\setlength\itemsep{0.3em}
  \item[(i)] $X$ has property [P],
  \item[(ii)] the set $\{(x,x^*)\in \strexp{B_X}\times \SE(X)\colon \|x^*\|=x^*(x)=1\}$ is dense in $\Pi(X)$,
  \item[(iii)] $X$ has property [P] witnessed with the function $\eps\longmapsto \eps^2/2$.
\end{enumerate}
\end{lemma}

\begin{proof}
Only that (ii) implies (iii) has to be proved. Pick $(x,x^*)\in S_X\times S_{X^*}$ such that $\re x^*(x)>1-\eps^2/2$ and apply the Bishop-Phelps-Bollob\'{a}s theorem (see \cite[Corollary~2.4]{C-K-M-M-R} for this version) to find $(y,y^*)\in \Pi(X)$ such that $\|y-x\|<\eps$ and $\|y^*-x^*\|<\eps$. Assertion (ii) allows us to find $z\in \strexp{B_X}$ and $z^*\in \SE(X)$ with $\|z^*\|=z^*(z)=1$ and satisfying that $\|z-x\|<\eps$ and $\|z^*-x^*\|<\eps$.
\end{proof}

We do not know if the separate density of $\strexp{B_X}$ in $S_X$ and that of $\SE(X)$ in $X^*$ implies property $[P]$, but the following result provides an useful sufficient condition.

\begin{proposition}\label{prop:ALUR}
Let $X$ be a Banach space. If $S_X = \strexp{B_X}$, then $X$ has property [P].
\end{proposition}

\begin{proof}
Take $(x_0,x_0^*)\in \Pi(X)$. As $x_0\in \strexp{B_X}$, there is $u_0^*\in \SE(X)$ which strongly exposed $x_0$. Since $x_0^*(x_0)=1$, it is immediate that the norm-one functional $x_n^*=(1+n^{-1}\|u_0^*\|)\bigl(x_0^* + n^{-1}u_0^*\bigr)$ strongly exposes $x_0$ for every $n\in \N$ and that $\{x_n^*\}\longrightarrow x_0^*$. Now, Lemma~\ref{lemma:charac-property-P} gives the result.
\end{proof}

Recall from \cite{BHLT2000} that a point $x$ in the unit sphere $S_X$ of a Banach pace $X$ is said to \emph{an almost LUR} (in short, \emph{ALUR}) point if any $(x_n) \subseteq B_X$ and $(x_m^*) \subseteq B_{X^*}$, the condition
\[
\lim_m \lim_n x_m^* \left (\frac{x_n + x}{2}\right) = 1
\]
implies that $\|x_n - x \| \longrightarrow 0$. We say that $X$ is \emph{ALUR} if every element of $S_X$ is ALUR. It is clear that LUR spaces are ALUR, but the reverse implication is not true (see \cite[Corollary 12]{BHLT2000}). It is observed in \cite[Corollary~4.6]{BHL2004} that if $X$ is ALUR, then each point $x$ in $S_X$ is strongly exposed by every $x^* \in S_{X^*}$ which attains its norm at $x$. Thus, in particular, if $X$ is ALUR, then $S_X = \strexp{B_X}$.

\begin{corollary}
ALUR Banach spaces satisfy property [P].
\end{corollary}

Our next aim is to provide a very general result in which the property [P] of a Banach space $X$ produce a denseness result of $\ASE(X,Y)$ which recall the Bishop-Phelps-Bollob\'{a}s property.

\begin{theorem}\label{thm:ACK}
Let $X$ be a Banach space with property [P], $Y$ be a Banach space with ACK$_\rho$ structure with the corresponding $1$-norming set $\Gamma \subseteq B_{Y^*}$. Then, there exists a function $\eps\in (0,1)\longmapsto \eta(\eps, \rho) >0$ such that if $T \in \flat(X,Y)$ satisfies that $\|T\| = 1$ and $\|Tx_0 \| > 1 -\eta(\eps, \rho)$ for some $x_0 \in S_X$, then there exists $S \in \ASE(X,Y)$ and $u_0 \in \strexp{B_X}$ such that $\|Su_0 \| = \|S \| =1$, $\|S-T\| < \eps$, and $\|u_0 - x_0\| < \eps$.
\end{theorem}

We need the following lemma which can be obtained by arguing as in \cite[Lemma~2.9]{CGKS2018} but using property [P] instead of the Bishop-Phelps-Bollob\'as theorem.

\begin{lemma}\label{lem:cgks}
Let $X$ be a Banach space which has property [P] with a function $\eps \longmapsto \eta(\eps)$, and let $Y$ be Banach space. Let $\Gamma \subseteq B_{Y^*}$ be a $1$-norming set, $T \in \Lin (X,Y)$ be a $\Gamma$-flat operator with $\|T\|=1$, $\eps >0$ and $x_0 \in S_X$ such that $\|T(x_0)\| > 1-\eta(\eps)$. Then for every $r>0$, there exist
\begin{enumerate}
\setlength\itemsep{0.3em}
\item $w^*$-open set $U_r \subseteq Y^*$ with $U_r \cap \Gamma \neq \emptyset$,
\item $x_r^* \in \SE(X)$ and $u_r \in \strexp{B_{X}}$ such that $|x_r^* (u_r)|= 1$, $\|T^* z^* - x_r^* \| < r+\eps+\eta(\eps)$, and $\|u_r - x_0\| < \eps$ for every $z^* \in U_r \cap \Gamma$.
\end{enumerate}
\end{lemma}

\begin{proof}[Proof of Theorem~\ref{thm:ACK}]
Given $\eps >0$, let $\eta(\eps)$ be the constant from the property [P]. Fix $0 < \eps_0 < \eps$ and take $\eps_1 >0$ such that
\[
\max \left\{ \eps_1, 4\left( \eps_1 + \eta(\eps_1) + \frac{2(\eps_1 + \eta(\eps_1))}{1-\rho + \eps_1 + \eta(\eps_1)}\right) \right\} < \eps_0.
\]
Take $r> 0$ and $0<\eps_2 <\frac{\eps}{3}$ so that $3\eps_2 + r < \eps_1 + \eta(\eps_1)$.

Now, let $T \in \Lin (X, Y)$ be a $\Gamma$-flat operator such that $\|T\| = 1$ and $\|T(x_0)\|>1-\eta(\eps_1)$ for some $x_0 \in S_X$. By Lemma \ref{lem:cgks}, there exists
\begin{enumerate}
\setlength\itemsep{0.3em}
\item $w^*$-open set $U_r \subseteq Y^*$ with $U_r \cap \Gamma \neq \emptyset$,
\item $x_r^* \in \SE(X)$ and $u_r \in S_X$ such that $|x_r^* (u_r)|= 1$, $\|u_r - x_0\| < \eps_1 $ and $\|T^* z^* - x_r^* \| < r+\eps_1 +\eta(\eps_1 )$, for every $z^* \in U_r \cap \Gamma$.
\end{enumerate}
On the other hand, by definition of ACK$_\rho$, we can obtain $V \subseteq U_r \cap \Gamma, y_1^* \in V, e \in S_Y, F \in \Lin (X,Y)$ and $V_1 \subseteq \Gamma$ satisfying the desired properties.

Define $S(x) := x_r^* (x) F(e) + (1-\delta)(\id_Y -F) T(x)$ for every $x \in X$, where $\delta \in (\eps_2, 1)$ is chosen so that $\|S\| \leq 1$ (it is possible to find such $\delta$, see \cite[Lemma 3.5]{CGKS2018}). Note that
\begin{align*}
1=|x_r^* (u_r)|= \| y_1^* (x_r^* (u_r)) F(e) \| = |y_1^* (S(u_r))| \leq \|S(u_r)\| \leq 1;
\end{align*}
which implies that $S$ attains its norm at $u_r$. Computing as in \cite[Lemma 3.5]{CGKS2018} (or, see \cite[Theorem 3.5]{CMstudia}), we have $\|S-T\| < \frac{\eps}{2}$. Finally, since $u_r$ is a strongly exposed point, by Lemma~\ref{lem:RMI}, there is $G \in \ASE( X,Y)$ such that $\|G(u_r)\| = 1$ and $\|G - S \| <\frac{\eps}{2}$, so $\|G-T\|<\eps$.
\end{proof}

\begin{proof}[Proof of Theorem~\ref{propo:Gammaflat-just-denseness}]
The idea is just to follow the proof of Theorem~\ref{thm:ACK}, forgetting the estimation on the distance between vectors in the domain space and then property [P] can be easily replaced by the density of $\SE(X)$ instead.
\end{proof}

\subsection{When the set of strongly exposed points in the range space is countable (up to rotations)}\label{subsection:32-countable}

Our next aim is to provide results on denseness of absolutely strongly exposing operators for which even the denseness of the norm attaining operators was unknown.  Our first general result in this line is the following one from which we will get a number of corollaries.

\begin{theorem}\label{thr:abstract-countably-boundary}
Let $X$, $Y$ be Banach spaces and let $\mathcal{I}(X,Y)$ be a closed subspace of $\Lin(X,Y)$ containing rank one operators. Suppose that there is a sequence $\{y_n^*\}$ in $S_{Y^*}$   such that the set
$$
\mathcal{A}=\bigl\{T\in \mathcal{I}(X,Y)\colon \|T\|=\|T^*y_n^*\|\ \text{for some } n\in \N\bigr\}
$$
is residual in $\mathcal{I}(X,Y)$. Then:
\begin{enumerate}
\setlength\itemsep{0.3em}
  \item[(a)] if $\NA(X,\K)$ is residual, then $\NA(X,Y)\cap \mathcal{I}(X,Y)$ is residual in $\mathcal{I}(X,Y)$;
  \item[(b)] if $\SE(X)$ is dense, then $\ASE(X,Y)\cap \mathcal{I}(X,Y)$ is dense in $\mathcal{I}(X,Y)$.
\end{enumerate}
\end{theorem}

In the proof of Theorem \ref{thr:abstract-countably-boundary}, we will use the following easy result on residuality.

\begin{lemma}\label{lemma:residualinrange}
Let $Z$, $W$ be Banach spaces, let $\mathcal{I}'(Z,W)$ be a closed subspace of $\Lin(Z,W)$ and let $\{z_n\}$ be a sequence in $S_Z$. Suppose that for every $n\in \N$, the bounded linear operator $\Phi_n\colon \mathcal{I}'(Z,W)\longrightarrow W$  given by $\Phi_n(T)=T(z_n)$ for every $T\in \mathcal{I}'(Z,W)$ is onto. Then, for every residual set $D$ of $W$, the set
$$
\mathcal{B}=\bigl\{T\in \mathcal{I}'(X,Y)\colon T(z_n)\in D \text{ for all } n\in \N\bigr\}
$$
is residual in $\mathcal{I}'(X,Y)$.
\end{lemma}

\begin{proof}
Let $\{O_n\}$ be a sequence of dense open set of $W$ such that $\bigcap\nolimits_{m\in \N} O_m \subseteq D$. As $\Phi_n$ is bounded linear and onto, $\Phi_n$ is an open map. Moreover, $\Phi_n^{-1}(O_m)$ is open and dense in $\mathcal{I}'(Z,W)$. Indeed, the set is open by continuity; also, for every open subset $U$ of $\mathcal{I}'(Z,W)$ and every $n,m\in\mathbb N$, $\Phi_n(U)\cap O_m\neq \emptyset$ as $\Phi_n (U)$ is open and $O_m$ is dense in $W$; hence $\Phi_n^{-1}(O_m)\cap U\neq \emptyset$. Now, the set
$$
\bigcap_{n,m\in \N}\Phi_n^{-1}(O_m)
$$
is residual in $\mathcal{I}'(Z,W)$ and it is immediate that it is contained in $\mathcal{B}$.
\end{proof}

\begin{proof}[Proof of Theorem~\ref{thr:abstract-countably-boundary}]
Suppose first that $\NA (X,\mathbb{K})$ is residual in $X^*$. We apply Lemma~\ref{lemma:residualinrange} with $Z=Y^*$, $W=X^*$, $z_n=y_n^*$ for every $n\in \N$, $D$ a residual set contained in $\NA(X,\K)\subset W$, and
$$
\mathcal{I}'(Z,W)=\{T^*\in \Lin(Z,W)\colon T\in \mathcal{I}(X,Y)\},
$$
which is closed since it is isometrically isomorphic to $\mathcal{I}(X,Y)$. Moreover,  $\Phi_n(\mathcal{I}'(Z,W))=W$ for every $n\in \N$. Indeed, for every $x_0^*\in W=X^*$, define $T\in \mathcal{I}(X,Y)$ by $Tx=x_0^*(x)y_n$ where $y_n\in Y$ is a point at which $y_n^* (y_n)=1$. Observe that $T^*\in \mathcal{I}'(Z,W)$ and $\Phi_n(T^*)=T^*(y_n^*)=x_0^*$; hence $\Phi_n$ is surjective. Now, we can apply Lemma~\ref{lemma:residualinrange} to have that the set
$$
\mathcal{B}=\bigl\{T\in \mathcal{I}(X,Y)\colon T^*(y^*_n)\in D \text{ for all } n\in \N\bigr\}
$$
is residual in $\mathcal{I}(X,Y)\equiv \mathcal{I}'(Z,W)$. Therefore, $\mathcal{A}\cap \mathcal{B}$ is also residual, but this intersection is contained in the set
$$
\mathcal{C}=\bigl\{T\in \mathcal{I}(X,Y)\colon \|T\|=\|T^*y_n^*\|\text{ with } T^*y_n^*\in D \text{ for some } n\in \N\bigr\}
$$
which is a fortiori residual in $\mathcal{I}(X,Y)$. As $D\subset \NA(X,\K)$, it follows from Fact~\ref{fact:containinginNA(X,Y)} that $\mathcal{C}\subset \NA(X,Y) \cap \mathcal{I} (X, Y)$, getting the residuality of the latter set.

In the case when $\SE(X)$ is dense, we take the above set $D$ to be $\SE(X)$. Then the set $\mathcal{C}$ is contained in the closure of $\ASE(X,Y)\cap \mathcal{I}(X,Y)$ by Lemma~\ref{lem:RMI}, hence $\ASE(X,Y)\cap \mathcal{I}(X,Y)$ is dense in $\mathcal{I}(X,Y)$.
\end{proof}

We are ready to present the main consequences of Theorem~\ref{thr:abstract-countably-boundary}.

\begin{corollary}\label{coro:countably-1}
Let $X$ be a Banach space, let $Y$ be a Banach space with the RNP such that $\strexp{B_Y}$ is countable up to rotations, and let $\mathcal{I}(X,Y^*)$ a closed subspace of $\Lin(X,Y^*)$ containing rank one operators.
\begin{enumerate}
\setlength\itemsep{0.3em}
  \item[(a)] If $\NA(X,\K)$ is residual, then $\NA(X,Y^*)\cap \mathcal{I}(X,Y^*)$ is residual in $\mathcal{I}(X,Y^*)$.
  \item[(b)] If $\SE(X)$ is dense, then the elements of $\mathcal{I}(X,Y^*)$ at which the norm of $\Lin(X,Y^*)$ is Fr\'{e}chet-differentiable is dense in  $\mathcal{I}(X,Y^*)$; in particular, $\ASE(X,Y^*)\cap \mathcal{I}(X,Y^*)$ is dense in $\mathcal{I}(X,Y^*)$.
\end{enumerate}
\end{corollary}

\begin{proof}
Write $\Psi\colon \Lin(X,Y^*)\longrightarrow \Lin(Y,X^*)$ given by $\Psi(T)=T^*|_{Y}$ and observe that $\Psi$ is an isometric isomorphism. Let $\{y_n\}$ a sequence in $S_Y$ such that $\T\{y_n\colon n\in \N\}=\strexp{B_Y}$. Then, the set
$$
\bigl\{T\in\mathcal{I}(X,Y^*)\colon \|T\|=\| \Psi (T) (y_n)\| \text{ for some } n\in \N\bigr\}
$$
contains $$\Psi^{-1}\left(\ASE(Y,X^*)\cap \Psi\bigl(\mathcal{I}(X,Y^*)\bigr)\right)$$ which is residual in $\mathcal{I}(X,Y^*)$ by Bourgain-Stegall result as $Y$ has the RNP (see the item (a) of Proposition~\ref{propo:BourgainStegall}). The first assertion of the corollary now follows from the same argument as in the proof of Theorem~\ref{thr:abstract-countably-boundary}, where the Lemma \ref{lemma:residualinrange} is applied to $\Psi (\Lin (X, Y^*)) = \Lin (Y, X^*)$ . For the second assertion (b), if $\SE (X)$ is dense in $X^*$, then by taking $D$ to be $\SE(X)$, we have the denseness of the set
$$
\bigl\{T\in \Psi\bigl(\mathcal{I}(X,Y^*)\bigr)\colon \|T\|=\|Ty\| \text{ with } Ty\in \SE(X)\text{ for some } y\in \strexp{B_Y}\bigr\}
$$
as in Theorem \ref{thr:abstract-countably-boundary}. But then Lemma~\ref{lem:RMI} gives that the set
\begin{multline*}
\mathcal{C}=\bigl\{T\in\Psi\bigl(\mathcal{I}(X,Y^*)\bigr)\cap \ASE(Y,X^*)\colon\\
  \|T\|=\|Ty\| \text{ with } Ty\in \SE(X)\text{ for some } y\in \strexp{B_Y}\bigr\}
\end{multline*}
is actually dense in $\Psi\bigl(\mathcal{I}(X,Y^*)\bigr)$. Now, Proposition~\ref{prop:Heinrich} shows that the norm of $\Lin(Y,X^*)$ is Fr\'{e}chet-differentiable at all elements of $\mathcal{C}$ so, the norm of $\Lin(X,Y^*)$ is Fr\'{e}chet-differentiable at all element of $\Psi^{-1}(\mathcal{C})$, which is dense in $\mathcal{I}(X,Y^*)$. The denseness of $\ASE(X,Y^*)$ follows also from Proposition~\ref{prop:Heinrich}.
\end{proof}

A first immediate consequence of this corollary deals with finite-dimensional range spaces.

\begin{example}\label{example:finite_dimensional_countable}
Let $X$ be a Banach space and let $Y$ be a finite-dimensional Banach space such that $\ext{B_{Y^*}}$ is countable up to rotations.
\begin{enumerate}
\setlength\itemsep{0.3em}
  \item[(a)] If $\NA(X,\K)$ is residual, then $\NA(X,Y)$ is residual.
  \item[(b)] If $\SE(X)$ is dense, then Fr\'{e}chet-differentiability points in $\Lin(X,Y)$ are dense; in particular, $\ASE(X,Y)$ is dense in $\Lin(X,Y)$.
\end{enumerate}
\end{example}

\begin{remark}
It was observed in \cite[p.~414]{AAP2} that a finite-dimensional Banach space has property quasi-$\beta$ if and only if $E_Y=\ext{B_{Y^*}}/\sim$ is a discrete topological space.
It is clear that for a finite-dimensional Banach space $Y$, the hypothesis that $\ext{B_{Y^*}}$ is countable up to rotations is much weaker than the hypothesis that $E_Y$ is discrete; hence we can obtain more examples from Corollary \ref{coro:countably-1} than the ones which can be obtained via property quasi-$\beta$.
It might be worth mentioning that for a \emph{2-dimensional} Banach space $Y$, the space $E_Y$ is discrete if and only if $E_Y$ is finite (since the set $\ext{B_{Y^*}}$ is compact in this case).
\end{remark}

Another interesting consequence of Theorem~\ref{thr:abstract-countably-boundary} is the following one which looks similar to the previous corollary, but now the conditions stated in Corollary \ref{coro:countably-1} are assumed for a dual Banach space.

\begin{corollary}\label{coro:countably-2}
Let $X$ be a Banach space, let $Y$ be a Banach space, and let $\mathcal{I}(X,Y)$ a closed subspace of $\Lin(X,Y)$ containing rank one operators. Suppose that $Y^*$ has the RNP and $\strexp{B_{Y^*}}$ is countable up to rotations.
\begin{enumerate}
\setlength\itemsep{0.3em}
  \item[(a)] If $\NA(X,\K)$ is residual, then $\NA(X,Y)\cap \mathcal{I}(X,Y)$ is residual in $\mathcal{I}(X,Y)$.
  \item[(b)] If $\SE(X)$ is dense, then the elements of $\mathcal{I}(X,Y)$ at which the norm of $\Lin(X,Y)$ is Fr\'{e}chet-differentiable is dense in $\mathcal{I}(X,Y)$; in particular, $\ASE(X,Y)\cap \mathcal{I}(X,Y)$ is dense in $\mathcal{I}(X,Y)$.
\end{enumerate}
\end{corollary}

\begin{proof} Write $\Psi\colon \Lin(X,Y)\longrightarrow \Lin(Y^*,X^*)$ given by $\Psi(T)=T^*$ and observe that $\Psi$ is an isometric embedding. Let $\{y^*_n\}$ be a sequence in $S_{Y^*}$ such that $\T\{y^*_n\colon n\in \N\}=\strexp{B_{Y^*}} \subseteq \NA (Y, \mathbb{K})$. By Theorem \ref{thr:abstract-countably-boundary}, it suffices to show that
$$
\mathcal{A}:=\bigl\{T\in\mathcal{I}(X,Y)\colon \|T\|=\|T^*y^*_n\| \text{ for some } n\in \N\bigr\}
$$
is residual in $\mathcal{I} (X, Y)$. This is immediate since $\mathcal{A}$ contains the following set $$\Psi^{-1}\left(\ASE(Y^*,X^*)\cap \Psi\bigl(\mathcal{I}(X,Y)\bigr)\right),$$ which is residual in $\mathcal{I}(X,Y)$ by Bourgain-Stegall result (see the item (b) of Proposition~\ref{propo:BourgainStegall}). Note that the denseness of elements of $\mathcal{I}(X,Y)$ at which the norm of $\Lin(X,Y)$ is Fr\'{e}chet-differentiable follows from Proposition \ref{prop:Heinrich}.
\end{proof}

Corollary~\ref{coro:countably-2} gives the following particular case.

\begin{example}\label{example:predual-ell1}
Let $X$ be a Banach space, let $Y$ be a predual of $\ell_1$, and let $\mathcal{I}(X,Y)$ a closed subspace of $\Lin(X,Y)$ containing rank one operators.
\begin{enumerate}
\setlength\itemsep{0.3em}
  \item[(a)] If $\NA(X,\K)$ is residual, then $\NA(X,Y)\cap \mathcal{I}(X,Y)$ is residual in $\mathcal{I}(X,Y)$.
  \item[(b)] If $\SE(X)$ is dense, then the elements of $\mathcal{I}(X,Y)$ at which the norm of $\Lin(X,Y)$ is Fr\'{e}chet-differentiable is dense in $\mathcal{I}(X,Y)$; in particular, $\ASE(X,Y)\cap \mathcal{I}(X,Y)$ is dense in $\mathcal{I}(X,Y)$.
\end{enumerate}
\end{example}

As far as we know, the question of whether all preduals of $\ell_1$ have Lindenstrauss property B remains unsolved, so the above result provides new examples of pairs of Banach spaces for which the set of norm attaining operators is dense.

Another consequence of Theorem~\ref{thr:abstract-countably-boundary} is the following. Recall that a subset $B\subset S_{Y^*}$ is a \emph{James boundary} for $Y$ if for every $y\in Y$ there is $y^*\in B$ such that $|y^*(y)|=\|y\|$.

\begin{corollary}\label{coro:countableboundary-compact}
Let $X$ be a Banach space and let $Y$ be a Banach space admitting a countable James boundary.
\begin{enumerate}
\setlength\itemsep{0.3em}
  \item[(a)] If $\NA(X,\K)$ is residual, then $\NA(X,Y)\cap \cpt(X,Y)$ is residual in $\cpt(X,Y)$.
  \item[(b)] If $\SE(X)$ is dense, then $\ASE(X,Y)\cap \cpt(X,Y)$ is dense in $\cpt(X,Y)$.
\end{enumerate}
\end{corollary}

The proof requires the fact, which is easy to prove, that the adjoint of a compact operator between Banach spaces attains its norm at an element of a prefixed James boundary.

\begin{remark}\label{fact_cpt}
Let $X$, $Y$ be Banach spaces and let $B\subset S_{Y^*}$ be a James boundary. Then, given $T\in \cpt(X,Y)$, there is $y^*\in B$ such that $\|T^*y^*\|=\|T\|$. Indeed, as $\overline{T(B_X)}$ is compact in $Y$, there is $y_0\in \overline{T(B_X)}$ with $\|y_0\|=\|T\|$. Pick $y_0^*\in B$ such that $|y_0^*(y_0)|=\|y_0\|=\|T\|$ and observe that
\begin{align*}
\|T^*y_0^*\| &\geq  \sup_{x\in B_X}\bigl|[T^*y_0^*](x)\bigr|= \sup_{x\in B_X}\bigl|y_0^*(Tx)\bigr| \\
   & = \sup_{y\in \overline{T(B_X)}} |y_0^*(y)|\geq |y_0^*(y_0)|=\|T^*\|.
\end{align*}
\end{remark}

\begin{proof}[Proof of Corollary~\ref{coro:countableboundary-compact}]
Write $B=\{y_n^*\colon n\in \N\}$ for the countable James boundary for $Y$ and use the previous Remark \ref{fact_cpt} to show that the set
$$
\{T\in \cpt(X,Y)\colon \|T\|=\|T^*y_n^*\|\ \text{for some } n\in \N\bigr\}
$$
coincides with $\cpt(X,Y)$, so is trivially residual in $\cpt(X,Y)$. Then, Theorem~\ref{thr:abstract-countably-boundary} applies and gives the results.
\end{proof}

\begin{remark}\label{remark:countable-boundary-subspaces}
Observe that a closed subspace of a Banach space admitting a countable James boundary also admits a countable James boundary (just consider the restrictions of elements of the boundary to the subspace which attain their norm at the subspace). Therefore, the denseness results from Corollary~\ref{coro:countableboundary-compact} pass to closed subspaces. This is not common in the theory of norm attaining operators: observe that $\ell_\infty$ has property $\beta$, hence property $B$ and there are separable Banach spaces $X$ and $Y$ for which $\NA(X,Y)\cap \cpt(X,Y)$ is not dense in $\cpt(X,Y)$ \cite{martinjfa}. Besides, the space $c$, which has property $\beta$, is not polyhedral, so it contains a two-dimensional subspace with infinitely many extreme points in its dual ball, hence failing property quasi-$\beta$ (so it is not know if such a subspace has property B). In any case, every subspace of $c$ satisfies the conditions of Corollary~\ref{coro:countableboundary-compact}.
\end{remark}

The following consequence of Corollary~\ref{coro:countableboundary-compact} is specially interesting.

\begin{example}\label{example:separablepolyhedralcompact}
Let $X$ be a Banach space and let $Y$ be a separable polyhedral space.
\begin{enumerate}
\setlength\itemsep{0.3em}
  \item[(a)] If $\NA(X,\K)$ is residual, then $\NA(X,Y)\cap \cpt(X,Y)$ is residual in $\cpt(X,Y)$.
  \item[(b)] If $\SE(X)$ is dense, then the norm of $\Lin(X,Y)$ is Fr\'{e}chet-differentiability at a dense subset of $\cpt(X,Y)$; in particular, $\ASE(X,Y)\cap \cpt(X,Y)$ is dense in $\cpt(X,Y)$.
\end{enumerate}
\end{example}

\begin{proof}
Every separable polyhedral space $Y$ admits a countable James boundary  \cite[Theorem~1.4]{Fonf2000} and then Corollary~\ref{coro:countableboundary-compact} gives the result. Only the part related to the denseness of Fr\'{e}chet-differentiability points in case (b) does not follows directly from such corollary, so let us prove it. It is actually proved in \cite[Theorem~1.4]{Fonf2000} that the set $\wstrexp{B_{Y^*}}$ is countable and it is a James boundary for $Y$. Therefore, the proof of Corollary~\ref{coro:countableboundary-compact} shows that the set
$$
\mathcal{A}=\bigl\{T\in \cpt(X,Y)\colon \|T^*y^*\|=\|T^*\| \text{ and } T^*y^*\in \SE(X) \text{ for some } y^*\in \wstrexp{B_{Y^*}}  \bigr\}
$$
is dense in $\cpt(X,Y)$. By Lemma~\ref{lem:RMI}, for every $T\in\mathcal{A}$ and every $\eps>0$ there is $S\in \cpt(X,Y)$ such that $\|T-S\|<\eps$ and $S^*\in \ASE(Y^*,X^*)$, $\|S^*\|=\|S^*y_0^*\|$ with $S^*y_0^*\in \SE(X)$. Proposition~\ref{prop:Heinrich} shows that $S^*$ is a point of Fr\'{e}chet-differentiability of the norm of $\Lin(Y^*,X^*)$ so, a fortiori, a point of Fr\'{e}chet differentiability of the norm of $\Linw (Y^*,X^*)$ and, therefore, $S$ is a point of Fr\'{e}chet-differentiability of the norm of $\Lin(X,Y)$.
\end{proof}

\begin{remark}
It is known (and easy to prove) that for every Banach space $X$ and every polyhedral space with the approximation property, $\NA(X,Y)\cap \cpt(X,Y)$ is dense in $\cpt(X,Y)$ \cite[Corollary~4.5]{Martin2016}. As far as we know, whether the assumption for $Y$ to have the approximation property can be removed or not is an open question. The previous example shows that this is the case when $\NA(X,\K)$ is residual and $Y$ is separable.
\end{remark}

\begin{remark}
While every separable polyhedral space contains a countable James boundary (actually the set of $w^*$-strongly exposed points of its dual ball) which is the key to proving Example~\ref{example:separablepolyhedralcompact}, there are examples of polyhedral Banach spaces $Y$ for which $\ext{B_{Y^*}}$ is uncontable (even that it cannot be covered by a countable union of compact sets, see \cite{Livni}). We do not know whether Corollary~\ref{coro:countably-2} is applicable for these spaces, as we do not know how big is the set $\strexp{B_{Y^*}}$ for these examples.
\end{remark}

We can remove the separability hypothesis in Example~\ref{example:separablepolyhedralcompact} in the case of $\SE(X)$ dense, but the result gives less information.

\begin{example}\label{example:General_polyhedralcompact}
Let $X$ be a Banach space for which $\SE(X)$ is dense and let $Y$ be a polyhedral space. Then, $\ASE(X,Y)\cap \cpt(X,Y)$ is dense in $\cpt(X,Y)$.
\end{example}

\begin{proof}
Fix $T\in \cpt(X,Y)$ and $\eps>0$. As $Z=\overline{T(X)}$ is separable and polyhedral, it follows from Example~\ref{example:separablepolyhedralcompact} that there is $S\in  \ASE(X,Z)\cap \cpt(X,Z)$ for which $\|T'-S\|<\eps$, where $T'$ is just $T$ considered as
operator from $X$ to $Z$. Now, write $J\colon Z\longrightarrow Y$ for the canonical inclusion and observe that $\|T-JS\|= \|T'-S\|<\eps$ and  $JS\in \ASE(X,Y)\cap \cpt(X,Y)$.
\end{proof}

Another case in which Corollary~\ref{coro:countableboundary-compact} applies is for real almost-CL-spaces with separable dual. Recall that a Banach space $Y$ is an \emph{almost-CL-space} if its unit ball is the absolutely closed convex hull of every maximal convex subset of $S_Y$. Examples of almost-CL-spaces are $C(K)$ spaces and $L_1(\mu)$ spaces, among many others. We refer the reader to \cite{MartPaya-CL} and references therein for more information on almost-CL-spaces.

\begin{example}\label{example:separable-almostCL}
Let $X$ be a Banach space and let $Z$ be a \emph{real} almost-CL-space with $Z^*$ separable and let $Y$ be a closed subspace of $Z$.
\begin{enumerate}
\setlength\itemsep{0.3em}
  \item[(a)] If $\NA(X,\K)$ is residual, then $\NA(X,Y)\cap \cpt(X,Y)$ is residual in $\cpt(X,Y)$.
  \item[(b)] If $\SE(X)$ is dense, then $\ASE(X,Y)\cap \cpt(X,Y)$ is dense in $\cpt(X,Y)$.
\end{enumerate}
\end{example}

\begin{proof}
It follows from \cite[Lemma 3]{MartPaya-CL} that $Z$ admits a countable James boundary (see the proof of \cite[Theorem~5]{MartPaya-CL} for details). Therefore, Corollary~\ref{coro:countableboundary-compact} and Remark~\ref{remark:countable-boundary-subspaces} apply.
\end{proof}

It is not know whether all subspaces of a real almost-CL-space with separable dual have Lindenstrauss property B. It is easy to find such subspaces failing property quasi-$\beta$: a two-dimensional subspace of $c$ with infinitely many extreme points in the dual ball.

The validity of a complex version of Example~\ref{example:separable-almostCL} is not clear. For instance, it is not true that complex almost-CL-spaces with separable dual contains a countable James boundary, as it can be checked from the two-dimensional $\ell_1$ space. As far as we know, it is an open problem if this space has property B. On the other hand, for $C(K)$ spaces the result is also valid in the complex case. Recall that a topological space is called \emph{scattered} if every subset of it contains an isolated point (relative to the subset).

\begin{example}\label{example:scattered}
Let $K$ be a Hausdorff scattered compact topological space, and let $Y$ be a closed subspace of (the real or complex space) $C(K)$. If $\SE(X)$ is dense, then $\ASE(X,Y)\cap \cpt(X,Y)$ is dense in $\cpt(X,Y)$.
\end{example}

\begin{proof}
Fix $T\in \cpt(X,Y)$ and $\eps>0$. Observe that $Z=\overline{T(X)}$ is separable and then there is countable compact space $K_T$ such that $Z$ is contained in $C(K_T)$ \cite[Theorem 2]{PelSem}. As $\{\delta_t\colon t\in K_T\}$ is clearly a James boundary for $C(K_T)$ and it is countable, it follows that $Z$ admits a countable James boundary. We can now argue as in the proof of Example~\ref{example:General_polyhedralcompact}.
\end{proof}

\begin{remark}\label{remark:C(K)space}
It is known \cite{Lin} that $C(K)$ spaces have property $\beta$ when $K$ contains a dense subset of isolated points (in particular, when $K$ is scattered); hence Theorem \ref{thm:quasibeta} can be applied $Y=C(K)$. Moreover, for a compact space $K$, since the dual of $C(K)$ is isometric to $L^1 (\mu)$ for some suitable measure $\mu$, the result in Example \ref{thm:cpt} is also valid for $Y = C(K)$.
The main interest of Example~\ref{example:scattered} is to show the denseness of $\ASE(X,Y)\cap \cpt(X,Y)$ when $Y$ is a \emph{closed subspace} of $C(K)$ spaces provided that $K$ is scattered.
\end{remark}

\subsection{When the set of strongly exposed points in the range space is discrete (up to rotations)}\label{subsection33-discrete}

Our next aim is to provide results on the residuality of $\NA(X,Y)$ which can be applied for non-separable $Y$'s. Instead of requiring countability of some sets as in Theorem~\ref{thr:abstract-countably-boundary} and its consequences, we will require some topological discreteness. Our first result in this line is the following one. We will use the following notation: a subset $A$ of $B_X$ is \emph{discrete up to rotations} if every sequence $\{a_n\}$ of elements of $A$ which converges in norm to an element $a\in A$ satisfies that $a_n=\theta_n a$ with $\theta_n\in \T$ for all sufficiently large $n$ (and then $\{\theta_n\}$ converges to $1$). In the real case, this is the same as requiring $A$ to be discrete for the norm topology.

\begin{theorem}\label{thm:RNP}
Let $X$ be a Banach space and let $Y$ be a Banach space with the RNP such that $\strexp{B_Y}$ is discrete up to rotations. Let $\mathcal{I}(X,Y^*)$ be a closed subspace of $\Lin(X,Y^*)$ containing rank one operators.
\begin{enumerate}
\setlength\itemsep{0.3em}
  \item[(a)] If $\NA(X,\K)$ is residual, then $\NA(X,Y^*)\cap \mathcal{I}(X,Y^*)$ is residual in $\mathcal{I}(X,Y^*)$.
  \item[(b)] If $\SE(X)$ is dense in $X^*$, then the elements of $\mathcal{I}(X,Y^*)$ at which the norm of $\Lin(X,Y^*)$ is Fr\'{e}chet differentiable are dense in $\mathcal{I}(X,Y^*)$. In particular, $\ASE(X,Y^*)\cap \mathcal{I}(X,Y^*)$ is dense in $\mathcal{I}(X,Y^*)$.
\end{enumerate}
\end{theorem}

We need the following easy lemma which will be used in the proof of Theorem \ref{thm:RNP}.

\begin{lemma}\label{lem:Bourgain}
Let $Z$, $W$ be Banach spaces, let $E\subseteq W$ be a dense subset, and let $\mathcal{I}'(Z,W)$ be a closed subspace of $\Lin(Z,W)$ such that for every $z\in S_Z$, the bounded linear operator $\Phi_z\colon \mathcal{I}'(Z,W)\longrightarrow W$ given by $\Phi_z(S)=S(z)$ is surjective. Then, given $T \in \ASE(Z, W)\cap \mathcal{I}'(Z,W)$ which absolutely strongly exposes $z_0\in S_Z$ and given $\varepsilon>0$, there exists $G\in\mathcal{I}'(Z,W)$ satisfying:
\begin{enumerate}
\setlength\itemsep{0.3em}
\item $\Vert T-G\Vert<\varepsilon$,
\item $G(z_0)\in E$,
\item There exists $\delta>0$ so that
$$
z\in B_Z \text{ satisfies }\Vert G(z)\Vert> \|G\| -\delta\ \Longrightarrow\ z\in \mathbb{T} B(z_0, \eps).
$$
\end{enumerate}
\end{lemma}

\begin{proof}
By hypothesis, $\Phi_{z_0}$ is onto, hence open. Define $\mathcal A$ to be the set of those $S\in \mathcal{I}'(Z,W)$ such that $\Vert T-S\Vert<\varepsilon$ and satisfying that there exists $\delta>0$ with the property
$$
z\in B_Z,\  \ \Vert S(z)\Vert> \|S\| -\delta\ \Longrightarrow\ z\in \mathbb{T} B(z_0, \eps).
$$
It is not difficult to prove that $\mathcal A$ is open and, clearly, $T\in \mathcal A$. Hence $\Phi_{z_0}(\mathcal A)$ is a non-empty open subset of $W$. Consequently, $\Phi_{z_0}(\mathcal A)\cap E\neq\emptyset$ and so, $\Phi_{z_0}^{-1}(E)\cap \mathcal A\neq \emptyset$.
\end{proof}

\begin{proof}[Proof of Theorem~\ref{thm:RNP}]
Notice that $\Lin(X,Y^*)$ is isometrically isomorphic to $\Lin(Y,X^*)$ through the surjective isometry $T\longmapsto \Psi(T) :=T^*|_{Y}$ for every $T\in \Lin(X, Y^*)$. Let $D= \bigcap_{n \in\N} O_n$, where $O_n$ is open and dense, be a $G_\delta$ dense subset of $\NA(X,\K)$. As $\T \NA(X,\K)=\NA(X,\K)$, we may and do suppose that $\T O_n=O_n$ for every $n\in \N$.
We claim that the set
\[
\mathcal{C}:=\bigl\{T\in \ASE (Y,X^*)\cap \Psi(\mathcal{I}(X,Y))\colon T\text{ attains its norm at some $y_0 \in B_Y$ with $Ty_0\in D$}\bigr\}
\]
is residual in $\Psi(\mathcal{I}(X,Y))$. Once the claim is proved, the proof of the theorem finishes. Indeed, $\Psi^{-1}(\mathcal{C})\subseteq \NA(X,Y^*)$ by Fact~\ref{fact:containinginNA(X,Y)}. Since $\Psi^{-1}(\mathcal{C})$ forms a $G_\delta$ dense set, we conclude that $\NA(X,Y^*)$ is residual in $\Lin (X,Y^*)$. If, moreover, $\SE(X)$ is dense in $X^*$, we may take $D=\SE(X)$ and Proposition~\ref{prop:Heinrich} shows that the norm of $\Lin(Y,X^*)$ is Fr\'{e}chet differentiable at every element of $\mathcal{C}$, hence the norm of $\Lin(X,Y^*)$ is Fr\'{e}chet differentiable at every element of $\Psi^{-1}(\mathcal{C})$, which is dense in $\mathcal{I}(X,Y^*)$.

Let us go to prove that the set $\mathcal{C}$ is residual. For each $n \in \N$, define $\mathcal A_n$ to be the set of those $T\in \Psi(\mathcal{I}(X,Y^*))\subset \Lin(Y,X^*)$ with the property that there exists a strongly exposed point $z_0\in B_Y$ such that $T(z_0)\in \bigcap_{k=1}^n O_k$ and that there exists $\delta>0$ satisfying that
\[
\{z\in B_Y \colon \Vert T(z)\Vert>\Vert T\Vert-\delta\} \subseteq \mathbb T B\left(z_0,1/n\right).
\]

\emph{Claim:} {$\mathcal A_n$ is open for every $n \in \N$}. Indeed, given $T\in \mathcal A_n$, take $z_0$ and $\delta>0$ witnessing the defining property of $\mathcal A_n$. Since $T(z_0)\in \bigcap_{k=1}^n O_k$, there exists $r>0$ such that $B(T(x_0),r)\subseteq \bigcap_{k=1}^n O_k$. Take $0< \delta'<\delta$ and choose $\mu<\min\{r,\frac{\delta-\delta'}{2}\}$. Now, if $\Vert G-T\Vert<\mu$ and $G\in \Psi(\mathcal{I}(X,Y^*))$, then
$$
\{z\in B_Y \colon \Vert G(z)\Vert>\Vert G\Vert-\delta'\}\subseteq \{z\in B_Y \colon \Vert T(z)\Vert>\Vert T\Vert-\delta\}\subseteq \mathbb T B\left(z_0,1/n\right).
$$
Besides, $\Vert T(z_0)-G(z_0)\Vert\leq r$, from where $G(z_0)\in \bigcap_{k=1}^n O_k$. This proves that $B(T,\mu)\subseteq \mathcal A_n$.

\emph{Claim:} $\mathcal A_n$ is dense in $\Psi(\mathcal{I}(X,Y^*))$ for each $n \in \N$. To this end, let $G \in \Psi(\mathcal{I}(X,Y^*))\subset \Lin (Y, X^*)$ and $n \in \N$ be fixed. Since $Y$ has the RNP, by Bourgain-Stegall result (see Proposition~\ref{propo:BourgainStegall}), there exists $T \in \ASE(Y, X^*)\cap \Psi(\mathcal{I}(X,Y^*))$ such that $\Vert G-T\Vert< \frac{1}{n}$. By applying Lemma~\ref{lem:Bourgain} for $Z=Y$, $W=X^*$, $E=\bigcap_{k=1}^n O_k$, $\mathcal{I}'(Z,W)=\Psi(\mathcal{I}(X,Y^*))$ and $T \in \ASE (Z, W) \cap \mathcal{I}' (Z, W)$, we can find an element $H \in \mathcal A_n$ with $\Vert T-H\Vert< \frac{1}{n}$, so $\Vert G-H\Vert<\frac{2}{n}$. This shows that $\mathcal A_n$ is $\frac{2}{n}$-dense in $\Psi(\mathcal{I}(X,Y^*))$ for every $n\in \N$. But observe that $\mathcal{A}_{n+1} \subseteq \mathcal{A}_n$ for every $n \in \N$, which implies that $\mathcal{A}_n$ is $\frac{2}{j}$-dense for any $j \geq n$. It follows that each $\mathcal{A}_n$ is actually dense in $\Lin (Y,X^*)$.

Therefore, $\mathcal A := \bigcap_{n\in\N} \mathcal A_n$ is a $G_\delta$ dense subset of $\Psi(\mathcal{I}(X,Y^*))$. Note that every element in $\mathcal A$ is an absolutely strongly exposing operator. Indeed, take $T\in \mathcal A$. Then, for every $n\in\mathbb N$, we may find a strongly exposed point $z_n\in B_Y$ with the property that there exists $\delta_n>0$ so that
\[
\{z\in B_Y\colon \Vert T(z)\Vert>\Vert T\Vert-\delta_n\}\subseteq \mathbb T B\left(z_n,\frac{1}{n}\right).
\]
It is immediate that $\Vert T(z_n)\Vert\longrightarrow 1$, from where the property defining $\mathcal A_n$ implies that there is a sequence $\{\theta_n\}$ in $\T$ such that $\{\theta_n z_n\}$ is a Cauchy sequence in $Y$. Since $Y$ is complete, we may take $z_0\in B_Y$ to be the limit of $(\theta_n z_n)$. It is immediate that $T$ absolutely strongly exposes $z_0$, hence $z_0\in \strexp{B_Y}$. By the discreteness assumption on the strongly exposed points of $B_Y$, we get that $z_0=\theta_n'z_n$ with $\theta_n'\in \T$ holds for every $n \geq n_0$ for suitable $n_0\in\mathbb N$. Consequently, we have that
$$
T(z_0)=T(\theta'_n z_n)\in\bigcap\limits_{k=1}^n \T O_k=\bigcap\limits_{k=1}^n O_k \text{ for every } n \geq n_0.
$$
By the arbitrariness of $n \in \N$, we conclude that $T(z_0)\in D$ which shows that $T\in \mathcal{C}$. Hence, the set $\mathcal{C}$ contains the $G_\delta$ dense subset $\mathcal{A}$ of $\Psi(\mathcal{I}(X,Y^*))$, finishing the proof.
\end{proof}

Notice that the above Theorem \ref{thm:RNP} is applicable to a closed subspace $\mathcal{I} (X, Y^*)$ of $\Lin (X,Y^*)$ when $Y = \ell_1 (\Gamma)$ for some set $\Gamma$. But, in this case, $Y^* = \ell_\infty (\Gamma)$ readily has property $\beta$, so the same result can be achieved from Theorem \ref{thm:quasibeta}.

The next result is somehow a dual version of the previous theorem, but the discreteness assumption on strongly exposed points is slightly different.

\begin{theorem}\label{thm:dualRNP}
Let $X$ be a Banach space and let $Y$ be a Banach space such that $Y^*$ has the RNP and that for every sequence $\{y_n^*\}$ of elements of $\wstrexp{B_{Y^*}}$ which converges to an element $y_0^*\in \strexp{B_{Y^*}}$, there is $n_0\in \N$ and a sequence $\{\theta_n\}$ in $\T$ such that $y_n^*=\theta_n y_0^*$ for every $n\geq n_0$. Let $\mathcal{I}(X,Y)$ be a closed subspace of $\Lin(X,Y)$ containing rank one operators.
\begin{enumerate}
\setlength\itemsep{0.3em}
  \item[(a)] If $\NA(X,\K)$ is residual, then $\NA(X,Y)\cap \mathcal{I}(X,Y)$ is residual in $\mathcal{I}(X,Y)$.
  \item[(b)] If $\SE(X)$ is dense in $X^*$, then the elements of $\mathcal{I}(X,Y)$ at which the norm of $\Lin(X,Y)$ is Fr\'{e}chet differentiable are dense in $\mathcal{I}(X,Y)$. In particular, $\ASE(X,Y)\cap \mathcal{I}(X,Y)$ is dense in $\mathcal{I}(X,Y)$.
\end{enumerate}
\end{theorem}

\begin{proof}
The idea of the proof is more or less similar to the one of Theorem \ref{thm:RNP}. We give a proof for the sake of completeness.
For each $\eps >0$, let us consider the set
\[
\mathcal{A}_\eps := \{ G \in \mathcal I(X,Y)\colon \exists \, \eta >0 \text{ and } y_0^* \in \strexp{B_{Y^*}} \text{ so that } S(G^*, \eta) \subseteq \mathbb{T} B( y_0^*, \eps ) \},
\]
where $S(G^*, \eta) := \{ y^* \in B_{Y^*}\colon \|G^*(y^*) \| > \|G^* \| - \eta \}$.

\emph{Claim:} for each $\eps >0$, $\mathcal{A}_\eps$ is dense. To this end, use Proposition~\ref{propo:BourgainStegall} to get that $A_\varepsilon$ is dense. Indeed, given $T\in  \mathcal{I}(X,Y)$ and $\delta>0$, there is $0<\rho<\delta$, $x^*\in X^*$ and $y\in Y$ so that $S:=T+\rho x^*\otimes y$, which is an element of $\mathcal{I} (X,Y^*)$, enjoys that $S^*\in \ASE(Y^*,X^*)$. It is obvious that $S\in \mathcal{A}_\varepsilon$. Since $\Vert T-S\Vert<\delta$, we conclude the density.

\emph{Claim:} for each $\eps>0$, $\mathcal{A}_\eps$ is contained in
\[
\mathcal{B}_\eps:=\bigl\{ G \in \mathcal{I} (X,Y^*)\colon \exists \eta >0 \text{ and } z_0^* \in \wstrexp{B_{Y^*}} \text{ so that } S(G^*, \eta) \subseteq \mathbb{T} B( z_0^*, 2\eps ) \bigr\}.
\]
Indeed, let $G \in \mathcal{A}_\eps$ be given. Let $\eta >0$ and $y_0^* \in \strexp{B_{Y^*}}$ be such that
$S(G^*, \eta) \subseteq \mathbb{T} B( y_0^*, \eps)$. Since $Y^*$ has the RNP, we have that $B_{Y^*} = \overline{\text{conv}}^{w^*} (\wstrexp{B_{Y^*}})$. Thus, there exists $z_0^* \in \wstrexp{B_{Y^*}}$ such that $z_0^* \in S(G^*, \eta)$. Find $\lambda \in \mathbb{T}$ so that $\| z_0^* - \lambda y_0^* \| < \eps$. Now, if $y^* \in S(G^*, \eta)$, then $$\|y^* - \mu \overline{\lambda} z_0^* \| \leq \| y^* - \mu y_0^* \| + \| \mu y_0^* - \mu \overline{\lambda} z_0^* \| < 2\eps$$ for some $\mu \in \mathbb{T}$. This implies that $S(G^*, \eta) \subseteq \mathbb{T} B(z_0^*, 2\eps)$; hence $G\in \mathcal{B}_\eps$.

Let $D= \bigcap_{n \in \N} O_n$ be a $G_\delta$ subset of $\NA (X, \mathbb{K})$ which is dense in $X^*$. Without loss of generality, we may assume that $\mathbb{T} O_n = O_n$ for every $n \in \N$. For each $n\in \N$, define $\mathcal{C}_n$ to be the set of those $T \in \mathcal I(X,Y)$ with the property that there exists $\eta >0$ and $z^*_0 \in \wstrexp{B_{Y^*}}$ such that $S(T^*, \eta) \subseteq \mathbb{T} B(z_0^*, \frac{2}{n})$ and $T^* (z_0^*) \in \bigcap_{j=1}^n O_j$.

\emph{Claim:} $\mathcal{C}_n$ is open and dense for every $n\in \N$. Clearly, $\mathcal C_n$ is open for every $n\in\mathbb N$, and the proof follows from the idea of the first Claim in the proof of Theorem~\ref{thm:RNP}. Let us prove that $\mathcal C_n$ is dense.
Indeed, let $G \in \mathcal I(X,Y)$ and $n \in \N$ be fixed. Then, by the previous claim, there is $\tilde{G} \in \mathcal{B}_{n^{-1}}$ witnessed by $\eta >0$ and $z_0 \in \wstrexp{B_Y^*}$ such that $\|G- \tilde{G} \| < \frac{1}{n}$. If we consider the set $\mathcal{A}$ of elements $T \in \mathcal I(X,Y)$ such that $\|T- G \| < \frac{2}{n}$ and satisfying that there exists $\delta >0$ so that $S(T^*, \delta) \subseteq \mathbb{T} B(z_0^*, \frac{2}{n})$, it is clear that $\tilde{G} \in \mathcal{A}$; hence $\mathcal{A}$ is non-empty. By considering the map $\Phi\colon \mathcal I(X,Y) \longrightarrow X^*$ given as $\Phi (T)=T^*(z_0^*)$ we conclude, with the same argument to that of Lemma \ref{lem:Bourgain}, we can observe that there exists $H \in \mathcal I(X,Y)$ satisfying $\|H- G \| < \frac{2}{n}$, $H^*(z_0^*) \in \bigcap_{j=1}^n O_j$ and there exists $\delta >0$ so that $S(H^*, \delta) \subseteq \mathbb{T} B(z_0^*, \frac{2}{n})$. This shows that each $\mathcal{C}_n$ is $\frac{2}{n}$-dense. As the sequence $\{\mathcal{C}_n\}$ is decreasing, it actually follows that each $\mathcal{C}_n$ is dense.

Finally, $\bigcap_n \mathcal{C}_{n}$ is a $G_\delta$ dense subset in $\mathcal I(X,Y)$. If $G \in \bigcap_n \mathcal{C}_{n}$, then there would be $z_0^* \in \strexp{B_{Y^*}}$ and a sequence $\{z_n^*\} \subseteq \wstrexp{B_{Y^*}}$ converging to $z_0^*$ such that $\|G^* (z_0^*) \| = \|G^*\|$ and $G^*(z_n^*) \in \bigcap_{j=1}^n O_j$ for every $n \in \mathbb{N}$. By our assumption, there is a sequence $\{ \theta_n\}$ in $\mathbb{T}$ such that $z_n^* = \theta_n z_0^*$ for all sufficiently large $n$. Consequently, we obtain that $G^* (z_0^*) \in \bigcap_{j=1}^n O_j$ for all large $n$ which implies that $G^* (z_0^*) \in D$. Finally, by taking pre-adjoint, we conclude that $\NA (X, Y)$ is a residual set in $\mathcal{L}(X,Y)$.
\end{proof}

Note that Theorem~\ref{thm:dualRNP} applies to isometric preduals of $\ell_1(\Gamma)$ as, clearly, the set of extreme points of its dual unit ball is discrete up to rotations.

\begin{example}\label{example:ell_1_gamma}
Let $X$ be a Banach space and let $Y$ be a predual of $\ell_1(\Gamma)$ for some set $\Gamma$. Let $\mathcal{I}(X,Y)$ be a closed subspace of $\Lin(X,Y)$ containing rank one operators.
\begin{enumerate}
\setlength\itemsep{0.3em}
  \item[(a)] If $\NA(X,\K)$ is residual, then $\NA(X,Y)\cap \mathcal{I}(X,Y)$ is residual in $\mathcal{I}(X,Y)$.
  \item[(b)] If $\SE(X)$ is dense in $X^*$, then the elements of $\mathcal{I}(X,Y)$ at which the norm of $\Lin(X,Y)$ is Fr\'{e}chet differentiable are dense in $\mathcal{I}(X,Y)$. In particular, $\ASE(X,Y)\cap \mathcal{I}(X,Y)$ is dense in $\mathcal{I}(X,Y)$.
\end{enumerate}
\end{example}

Beside the case when $Y$ is either $\ell_1 (\Gamma)$ or an isometric predual of $\ell_1(\Gamma)$, some applications of Theorem \ref{thm:RNP} and \ref{thm:dualRNP} to the setting of the Lipschitz-free space over a certain metric space will be provided in Section~\ref{sect:applicationsLipschitz-bilinear}.

\section{Residuality and Fr\'{e}chet differentiability in the space of operators}\label{section:residuality}
In this section we address the following natural question.
\begin{equation}
  \tag{Q3}\label{Q3}
  \parbox{\dimexpr\linewidth-4em}{%
    \strut
\slshape Does the residuality of $\NA(X,Y)$ imply the density of $\ASE(X,Y)$ in $\Lin (X,Y)$?
    \strut
  }
\end{equation}
In the case when $Y$ is one-dimensional, the residuality of $\NA(X, \mathbb{K})$ is closely related to the Fr\'echet differentiability of the dual norm on $X^*$ (and hence to the geometric structure of the unit ball of $X$ due to the \v{S}mulyan test). For instance, it has been shown by Guirao, Montesinos, and Zizler \cite[Theorem~3.1]{GMZ} that if $X$ is separable, then $\NA(X, \mathbb{K})$ is residual if and only if the dual norm on $X^*$ is Fr\'echet differentiable on a dense subset of $X^*$, hence if and only if $\SE(X)$ is dense in $X^*$ (by \cite[Corollary 1.5]{DGZ}). This result has been extended by Moors and Tan \cite{MoorsTan} showing that the same conclusion holds for \emph{dual differentiation} Banach spaces. Examples of dual differentiation spaces are those Banach space which can be equivalently renormed to be LUR \cite{GKe} and also Banach spaces whose duals are weak Asplund (in particular, RNP spaces) \cite{GKMS}. It is actually an open question whether every Banach space is a dual differentiation space.

Under separability assumptions on $X$ and $Y^*$, the previous result of Guirao et al.\ can be extended to the case of bounded linear operators from $X$ to $Y$.

\begin{theorem}\label{thm:residual}
Let $X$ and $Y$ be Banach spaces. Suppose that $X$ and $Y^*$ are separable, and that $\NA(X,Y)$ is residual. Then, the points of $\Lin(X,Y)$ at which the norm is Fr\'echet differentiable are dense in $\Lin(X,Y)$, in particular, $\ASE (X, Y)$ is dense in $\Lin (X,Y)$.
\end{theorem}

\begin{proof}
Let $P$ be a $G_\delta$ subset of $\NA(X,Y)$ which is dense in $\Lin (X,Y)$. From the fact that a $G_\delta$ dense subset of a Baire space is again Baire, $P$ is a Baire space (see \cite[Theorem 589]{MZZ2015}, for instance). Let $A \subseteq P$ be the set of all points where the norm of $\Lin (X,Y)$ is not Fr\'echet differentiable.
We claim that $A$ is meager in $P$, which is enough to finish the proof. Note that the denseness of $\ASE(X,Y)$ in $\Lin (X,Y)$ is then immediate by  Proposition~\ref{prop:Heinrich}.

For each $T \in A$, there exists $x_T \in S_X$ where $T$ attains its norm. Take $y_T^* \in S_{Y^*}$ so that $y_T^* (T(x_T)) = \|T\|$. Note that for every $G \in \Lin (X,Y)$,
\[
\|T+G\| - \|T \| \geq \re [x_T \otimes y_T^*] (G),
\]
where the tensor $x_T \otimes y_T^*$ is considered as an element of $\Lin (X,Y)^*$. Using that $T$ is not a point of Fr\'echet differentiability, we can take $m_T \in \mathbb{N}$ such that
\[
\limsup_{G \rightarrow 0} \frac{ \|T+G\|-\|T\| - \re [x_T \otimes y_T^*](G)}{\|G\|} > \frac{1}{m_T}.
\]
For each $m \in \N$, let $A_m := \{ T \in A \colon m_T = m\}$ and consider a cover of $$B_X \otimes B_{Y^*} := \{x\otimes y^*\colon x \in B_X, y \in B_{Y^*} \} \subseteq \Lin (X,Y)^*$$ by open balls of radius $(12m)^{-1}$. Since $B_X \otimes B_{Y^*}$ is separable, by the Lindel\"of property, there is a countable subcover $\{B_k^m\}$ of open balls of radius $(12m)^{-1}$. For $m, k \in \N$, define $A_{m,k} : = \{T \in A_m\colon x_T \otimes y_T^* \in B_{k}^m\}$. Observe that $$\|x_T \otimes y_T^* - x_G \otimes y_G^*\| < (6m)^{-1}$$ for all $T, G \in A_{m,k}$. From the \v{S}mulyan test, $A = \bigcup_{m,k} A_{m,k}$. We claim that $A_{m,k}$ is nowhere dense in $P$ for each $m, k \in \N$, which will show that $A$ is meager in $P$, finishing then the proof of the theorem. Assume to the contrary that there is an nonempty open subset $O$ of $\overline{A_{m,k}}$ for some $m,k \in \N$. Pick $T \in O \cap A_{m,k}$ and $r >0$ such that $B(T,r) \subseteq O$. Since $T \in A_m$, we can find $H \in \Lin (X,Y)$ such that $\|H\| < \frac{r}{2}$ and
\[
\|T+H\|-\|T\| > \frac{\|H\|}{m} + (x_T \otimes y_T^* )(H).
\]
Since $P$ is dense in $\Lin (X,Y)$, we may assume that $T+H \in P$.

Note that $B\left(T+H, \frac{\|H\|}{3m}\right) \subseteq B(T,r) \subseteq O \subseteq \overline{A_{m,k}}$. Thus, $B\left(T+H, \frac{\|H\|}{3m}\right) \cap A_{m,k}$ is nonempty. Take $G \in B\left(T+H, \frac{\|H\|}{3m}\right) \cap A_{m,k}$. Observe that the following is true:
\begin{enumerate}
\setlength\itemsep{0.3em}
\item $\|T+H-G\| < \frac{\|H\|}{3m}$,
\item $\|T-G\| \leq \|H\| + \frac{\|H\|}{3m} < 2\|H\|$,
\item $\|T+H\| -\|G\| \leq \|T+H-G\| < \frac{\|H\|}{3m}$.  \label{eq:THG}
\end{enumerate}
Then,
\begin{align*}
\|T\|-\|G\| = \|G+(T-G)\|-\|G\| &\geq [x_G\otimes y_G^*](G+(T-G)) - [x_G\otimes y_G^*] (G) \\
&= [x_G \otimes y_G^*] (T-G).
\end{align*}
On the other hand,
\begin{align*}
\|T+H\| - \|G\| &= \|T+H\|-\|T\| +\|T\| -\|G\| \\
&> \frac{\|H\|}{m} + [x_T \otimes y_T^*] (H) + [x_G \otimes y_G^*) (T-G) \\
&= \frac{\|H\|}{m} + [x_T\otimes y_T^*](T+H-G) + [x_G\otimes y_G^* - x_T\otimes y_T^*](T-G) \\
&\geq \frac{\|H\|}{m} - \frac{\|H\|}{3m} - 2\|H\|\frac{1}{6m} = \frac{\|H\|}{3m},
\end{align*}
where we have used (1) and (2) in the last inequality. This contradicts (3). So, we conclude that $A_{m,k}$ is nowhere dense for each $m,k\in\N$.
\end{proof}

\begin{remark}
It is not possible, in general, to get denseness of Fr\'echet differentiable points of $\Lin (X,Y)$ when $\NA(X,Y)$ is residual. For instance, if $X$ has the Radon-Nikod\'ym property, then $\ASE (X, \ell_1)$ is dense in $\Lin (X,\ell_1)$ (hence, $\NA(X,\ell_1)$ is residual), but there is no point in $\Lin (X,\ell_1)$ where the norm on $\Lin (X,\ell_1)$ is Fr\'echet differentiable since the norm of $\ell_1$ is nowhere Fr\'echet differentiable (use Proposition~\ref{prop:Heinrich}). On the other hand, there is no known objection, as far as we know, to get the denseness of $\ASE(X,Y)$ from the residuality of $\NA(X,Y)$ in complete generality. On the other hand, we do not know if the residuality of $\NA(X,Y)$ for a non-trivial $Y$ implies that of $\NA(X,\K)$ as the denseness of $\ASE(X,Y)$ does.
\end{remark}

\begin{remark}\label{remark:residual}
\begin{enumerate}
\setlength\itemsep{0.3em}
\item
If the Banach space $Y$ in Theorem \ref{thm:residual} is \emph{reflexive}, then the denseness of the set of Fr\'echet differentiable points of $\Lin (X,Y)$ can be obtained directly from \cite[Theorem 3.1]{GMZ}. Indeed, in this case we have that $(X \pten Y^*)^*=\Lin(X,Y)$ and the inclusion $\NA (X,Y) \subseteq \NA(X \pten Y^*, \mathbb{K})$ holds. It follows that $\NA(X\pten Y^*,  \mathbb{K})$ is residual. Applying \cite[Theorem~3.1]{GMZ} to $X \pten Y^*$, the dual norm of $(X \pten Y^*)^* = \Lin (X,Y)$ is Fr\'echet differentiable on a $G_\delta$-dense subset.
\item
With the aid of the recent result \cite{AMRR} of A. Avil\'es et al., we can obtain a non separable version of Theorem \ref{thm:residual}. That is, if $X$ is a subspace of a WCG space, $Y$ is a reflexive Banach space, and $\Lin (X, Y) =\mathcal{K}(X, Y)$, then the residuality of $\NA(X,Y)$ implies the denseness of the points of $\Lin (X,Y)$ at which the norm is Fr\'echet differentiable. Indeed, the assumptions show that $X \pten Y^*$ is a subspace of WCG space \cite[Corollary 5.21]{AMRR}. Since WCG spaces are dual differentiation Banach spaces \cite{GKe} and a closed subspace of a dual differentiation Banach space is again a dual differentiation Banach space \cite{GKMS}, the space $X \pten Y^*$ turns to be a dual differentiation Banach space. Since $Y$ is reflexive, arguing as in the above item (1), the residuality of $\NA(X, Y)$ implies the residuality of $\NA (X \pten Y^*, \mathbb{K})$. Now, \cite[Theorem 3]{MoorsTan} proves that the dual norm of $(X \pten Y^*)^* = \Lin (X,Y)$ is Fr\'echet differentiable on a $G_\delta$-dense subset.
\end{enumerate}
\end{remark}

We next present two more related observations, one of them providing a partial solution to an open problem in \cite{ABR}.

\begin{remark}
\begin{enumerate}
\setlength\itemsep{0.3em}
\item
Given separable Banach spaces $X$ and $Y$, if $B_X$ is not dentable and $Y$ is reflexive, then $\NA(X,Y)$ is of the first Baire category. To see this, observe first that $B_{X \pten Y^*}$ is not dentable and it is separable. By a result of Bourgain and Stegall (see the proof of \cite[Theorem~3.5.5]{Bourgin}), $\NA(X \pten Y^*)$ turns to be of the first category in $(X\pten Y^*)^* = \Lin (X,Y)$. Since $\NA (X,Y) \subseteq \NA(X\pten Y^*)$ by the reflexivity of $Y$, we have that $\NA(X,Y)$ is also of the first category in $\Lin (X,Y)$. This observation gives a partial answer to \cite[Problem 7 in p.~12]{ABR}.
\item Let $X$ be a separable Banach space and let $Y$ be a separable reflexive space. If $X$ is convex-transitive and $\NA(X,Y)$ is of the second Baire category, then $X$ must be super-reflexive. Indeed, under the assumption, by the above item, $B_X$ must be dentable. It follows that the norm on $X^*$ is not rough \cite[Proposition 1.11]{DGZ}. Since $X$ is convex-transitive, $X^*$ is convex $w^*$-transitive; hence \cite[Theorem 3.2]{BR1999} implies that $X$ is super-reflexive. This extends the result in \cite{BR1999} that convex-transitive RNP spaces are super-reflexive.
\end{enumerate}
\end{remark}

\section{Applications to the geometry of Lipschitz-free spaces, to strongly norm attaining Lipschitz maps, and to norm attaining bilinear forms}\label{sect:applicationsLipschitz-bilinear}

\subsection{Lipschitz functions spaces and strong norm attainment}

Throughout this subsection, we will only consider real Banach spaces. Given a pointed metric space $M$ and a Banach space $Y$, the notation $\Lip (M, Y)$ denotes the Banach space of all Lipschitz maps $F\colon M \longrightarrow Y$ which vanishes at $0$, endowed with the Lipschitz norm given by
\[
\|F\|_L := \sup \left\{ \frac{\|F(x)-F(y)\|}{d(x,y)}\colon x,y \in M, x\neq y\right\}.
\]
Recall from \cite{CCGMR2019} that $F \in \Lip (M,Y)$ is said to \emph{strongly attain its norm} when the above supremum is actually a maximum, that is, when there exists $x \neq y$ in $M$ such that
\[
\|F\|_L =\frac{\|F(x)-F(y)\|}{d(x,y)}.
\]
By $\SA (M, Y)$, we denote the set of all strongly norm attaining Lipschitz maps in $\Lip (M, Y)$.

There is a connection between the strong norm attainment in spaces of Lipschitz functions and the classical norm attainment in spaces of operators. In order to exhibit it, we need to introduce a bit of notation. Let $M$ be a pointed metric space. We denote by $\delta$ the canonical isometric embedding of $M$ into $\Lip(M,\R)^*$, which is given by $\langle f, \delta(x) \rangle =f(x)$ for $x \in M$ and $f \in \Lip(M,\R)$. We denote by $\mathcal{F}(M)$ the norm-closed linear span of $\delta(M)$ in the dual space $\Lip(M,\R)^*$, which is usually called the \emph{Lipschitz-free space over $M$}, see the papers \cite{Godefroy-survey-2015} and \cite{gk}, and the book \cite{wea5} (where it receives the name of Arens-Eells space) for background on this. It is well known that $\mathcal{F}(M)$ is an isometric predual of the space $\Lip(M,\R)$ \cite[p.~91]{Godefroy-survey-2015}. When $M$ is a pointed metric space and $Y$ is a Banach space, every Lipschitz map $f \colon M \longrightarrow Y$ can be isometrically identified with the continuous linear operator $\widehat{f} \colon \mathcal{F}(M) \longrightarrow Y$ defined by $\widehat{f}(\delta_p)=f(p)$ for every $p \in M$. This mapping completely identifies the spaces $\Lip(M,Y)$ and $\mathcal{L}(\mathcal{F}(M),Y)$. Bearing this fact in mind, the set $\SA(M,Y)$ is identified with the set of those elements of $\mathcal{L}(\mathcal{F}(M),Y)$ which attain their operator norm at some \emph{molecule}, that is, at an element of $\mathcal{F}(M)$ of the form
\[
m_{x,y}:=\frac{\delta(x)-\delta(y)}{d(x,y)}
\]
for $x, y \in M$, $x \neq y$. We write $\Mol{M}$ to denote the set of all molecules of $M$. Note that, since $\Mol{M}$ is balanced and norming for $\Lip(M,\mathbb{R})$, a straightforward application of Hahn-Banach theorem implies that
$$
\overline{\co}(\Mol{M})=B_{\mathcal F(M)}.
$$
From this point of view, it is now clear that when $\SA(M,Y)$ is dense in $\Lip(M,Y)$, then $\NA(\mathcal{F}(M),Y)$ has to be dense in $\mathcal{L}(\mathcal{F}(M),Y)$ a fortiori. The converse result is not true as, for instance, $\NA(\mathcal{F}(M),\R)$ is always dense by the Bishop-Phelps theorem but there are many metric spaces $M$ such that $\SA(M,\R)$ is not dense in $\Lip(M,\R)$ \cite{kms}. See \cite{CCGMR2019,cgmr2020,kms} and references therein for background on the denseness of strongly norm attaining Lipschitz functions.

Of course, if $\SA(M,Y)$ is dense in $\Lip(M,Y)$ for every Banach space $Y$, then $\mathcal{F}(M)$ has property A. However, the question whether the property A of $\mathcal F(M)$ implies that  $\SA(M,Y)$ is dense in $\Lip(M,Y)$ for every Banach space $Y$ is one of the main questions in theory of strong norm attainment of Lipschitz functions (asked at \cite{CCGMR2019,cgmr2020}). It is even open the question whether Lindenstrauss property A of $\mathcal F(M)$ implies that $\SA(M,\mathbb R)$ dense in $\Lip(M,\mathbb R)$. As a consequence of Theorem \ref{thm:LURrenorming}, we obtain the following partial answer to this question.

\begin{corollary}\label{coro:corLipschitzfree}
If $M$ is a separable metric space and $\Free (M)$ has property A, then $\SA (M, \mathbb{R})$ is dense in $\Lip (M, \mathbb{R})$.
\end{corollary}

\begin{proof}
Since $M$ is separable, the Lipschitz-free space $\Free (M)$ over $M$ is separable; hence it admits an LUR renorming. Thus, by Corollary \ref{corollary:PropertyAimpliesSEdense}, $\SE (\Free (M))$ is dense in $\Free (M)^* = \Lip (M,\mathbb{R})$. Since every strongly exposed point of $\Free (M)$ is indeed a molecule \cite[Proposition 1.1]{CCGMR2019}, $\SE (\Free (M))\subset \SA(M,\mathbb{R})$, we conclude that $\SA (M, \mathbb{R})$ is dense in $\Lip (M, \mathbb{R})$.
\end{proof}

We do not know whether the separability assumption can be removed in the above result. Clearly, this assumption can be replaced by the hypothesis that $\Free(M)$ admits an LUR renorming. However, we do not know which non-separable metric spaces $M$ satisfy that $\Free(M)$ admits an LUR renorming.

Let us obtain consequences of Corollary~\ref{coro:corLipschitzfree}. First of all, consider the unit sphere $\mathbb{T}$ of the Euclidean plane endowed with the inherited Euclidean metric. It is shown in \cite[Theorem 2.1]{cgmr2020} that $\SA (\mathbb{T}, \mathbb{R})$ is not dense in $\Lip (\mathbb{T}, \mathbb{R})$, hence Corollary \ref{coro:corLipschitzfree} implies $\Free(\T)$ fails property A. On the other hand, it was also observed in \cite[Theorem 2.1]{cgmr2020} that every molecule of $\mathcal F(\mathbb T)$ is a strongly exposed point hence, in particular, $\cconv\bigl(\strexp{B_{\Free (\mathbb{T})}}\bigr)=B_{\Free (\mathbb{T})}$. So, we obtain the following result.

\begin{example}\label{example:Lindenstrauss-not-sufficient}
The separable space $\Free (\mathbb{T})$ fails to have property A, while $$B_{\Free (\mathbb{T})}=\cconv\bigl(\strexp{B_{\Free (\mathbb{T})}}\bigr).$$
\end{example}

This answer the implicit question from \cite{cgmr2020} of whether $\Free(\T)$ has property A. This question was discussed during the PhD defense of Rafael Chiclana in March 2021 and this was the starting point of the research conducting to the elaboration of the present manuscript.

\begin{remark}\label{remark:torusnotA}
\begin{enumerate}
\setlength\itemsep{0.3em}
\item The arguments before Example~\ref{example:Lindenstrauss-not-sufficient} show that $X=\mathcal F(\mathbb T)$ is an example of a separable Banach space where $B_X=\cconv(\strexp{B_X})$ but $\SE(X)$ is not dense in $X^*$ (and hence $X$ fails property A). As far as we know, a previous example of this kind has not been mentioned in the literature. Even more, the functionals in $X^*$ exposing $B_X$ are also not dense in $X^*$. Indeed, every functional exposing $B_X$ attains its norm at an exposed point, hence at a extreme point. Since extreme points of $B_{\Free(\mathbb{T})}$ are molecules \cite[Theorem 1]{aliaga}, functionals exposing $B_X$ are contained in $\SA(\T,\R)$ which is not dense.
\item The result of Example \ref{example:Lindenstrauss-not-sufficient} should be compared with the fact that if $B_X = \cconv(A)$ for a set $A$ of uniformly strongly exposed points, then $\ASE(X,Y)$ is dense in $\Lin (X,Y)$ for every $Y$, see \cite[Proposition~4.2]{cgmr2020}.
\end{enumerate}
\end{remark}

There are many other examples of separable metric spaces $M$ for which $\SA(M,\R)$ is not dense and hence, $\F{M}$ fails property A by using Corollary~\ref{coro:corLipschitzfree}: when $M$ is a length metric space  \cite[Theorem 2.2]{CCGMR2019}, in particular, when $M$ is a closed convex subset of a separable Banach space. But in all these cases, the unit ball of $\Free (M)$ fails to have strongly exposed points, so they fails property A by just using Lindenstrauss's necessary condition \cite[Theorem 2]{Lin}. New examples of metric spaces $M$ for which $\SA(M,\mathbb R)$ is not dense in $\Lip(M,\mathbb R)$ have appeared recently in \cite{chiclana2021}: every metric space $M$ which is the range of a $C^1$-curve into a Banach space whose derivative is not identically $0$. As a consequence of Corollary~\ref{coro:corLipschitzfree}, we have the following example.

\begin{example}
Let $M$ be the range of a $C^1$-curve into a Banach space whose derivative is not identically $0$. Then, $\Free(M)$ fails property A.
\end{example}

Next, we show examples of Lipschitz-free spaces so that the set of strongly exposed points is countable up to rotations or discrete up to rotations. This will enlarge the class of target spaces to which the results of Section \ref{section:sufficient} can be applied.

First, in the case of some countable metric spaces, the following result gives a case in which the results of Subsection \ref{subsection:32-countable} apply.

\begin{example}\label{exam:extremeLipfree1}
If $M$ is a countable proper (i.e.\ every closed ball is compact) metric space, then $\mathcal F(M)$ has the RNP \cite{dalet} and $\strexp{B_{\mathcal F(M)}}$ is countable up to rotations (indeed, $\Mol{M}$ is bijective with a subset of $M^2$, which is countable).\newline      Therefore, Corollary \ref{coro:countably-1} can be applied to $Y=\Free(M)$ getting, for instance, that $\ASE(X,Y^*)$ is dense in $\Lin(X,Y^*)$ for every Banach space $X$ such that $\SE(X)$ is dense in $X^*$.
\end{example}

In the cases covered by the previous example, the spaces $\Free(M)$ are actually dual spaces, so Corollary \ref{coro:countably-2} can be also applied for the preduals. We need some notation. The \emph{little Lipschitz space} on a metric space $M$ is the subspace $\lip(M)$ of $\Lip(M)$ of those functions $f$ satisfying that for every $\eps>0$, exits $\delta>0$ such that $|f(x)-f(y)|\leq \eps\,d(x,y)$ when $d(x,y)<\delta$. When $M$ is countable compact, $\lip(M)^*\equiv \Free(M)$ \cite{daletPAMS}. When $M$ is countable proper, the isometric predual of $\Free(M)$ is the following space (see \cite{dalet}:
$$
\mathcal{S}(M):= \left\{f\in \lip(M)\colon \lim_{r\to +\infty} \sup_{\begin{array}{c}
        x \text{ or } y \notin \overline{B}(0,r) \\
        x\neq y
      \end{array} } \frac{|f(x)-f(y)|}{d(x,y)}=0
\right\}
$$
(which coincides with $\lip(M)$ in the case when $M$ is compact).
As we already mentioned, Corollary \ref{coro:countably-2} can be applied to get the following result with the arguments in Example \ref{exam:extremeLipfree1}.

\begin{example}\label{exam:extremeLipfree3}
Let $X$ be a Banach space such that $\SE(X)$ is dense in $X^*$.
\begin{enumerate}
  \item[(a)] If $M$ is a countable compact metric space, then $\ASE(X,\lip(M))$ is dense in $\Lin(X,\lip(M))$.
  \item[(b)] If $M$ is a countable proper metric space, then $\ASE(X,\mathcal{S}(M))$ is dense in $\Lin(X,\mathcal{S}(M))$.
\end{enumerate}
\end{example}

Finally, the results of Subsection~\ref{subsection33-discrete} can be applied for discrete metric spaces.

\begin{example}\label{exam:extremeLipfree2}
If $M$ is a discrete metric space, then $\mathcal F(M)$ has the RNP and $\strexp{B_{\mathcal F(M)}}$ is discrete up to rotations. Indeed, it has the RNP by \cite[Theorem C]{agpp2021} and $\Mol{M}$ satisfies that if a net of molecules $m_{x_\alpha,y_\alpha}$ converges weakly to $m_{x,y}$, then $d(x_\alpha,x)\rightarrow 0$ and $d(y_\alpha,y)\rightarrow 0$ \cite[Lemma 2.2]{lppr}.\newline
   Therefore, Theorem \ref{thm:RNP} can be applied to $Y=\Free(M)$ getting, for instance, that $\ASE(X,Y^*)$ is dense in $\Lin(X,Y^*)$ for every Banach space $X$ such that $\SE(X)$ is dense in $X^*$.
\end{example}

Next, we would like to apply the results of Section~\ref{section:sufficient} to provide more examples of pairs $(M, Y)$ for which the set $\SA (M,Y)$ is dense in $\Lip (M, Y)$. Observe that, given a metric space $M$ and a Banach space $Y$, then $\ASE(\mathcal F(M),Y)$ is contained in $\SA(M,Y)$ since every element of $\ASE(\mathcal F(M),Y)$ attains its norm at a strongly exposed point of $\mathcal F(M)$, which is a molecule. As a consequence, we get the following result which extends previous results from \cite{CMstudia} and \cite{ChiMar-NA}. We have not included the results which are covered by these two references.

\begin{corollary}\label{coro:SNAdensefromA}
Let $M$ be a metric space in one of the following situations:
\begin{itemize}
\setlength\itemsep{0.3em}
\item [(a)] $M$ is separable and $\mathcal F(M)$ has property A,
\item [(b)] $M$ is a compact metric space not containing any isometric copy of $[0,1]$ and satisfying that $\SA(M,\mathbb R)$ is dense and $\Lip(M,\mathbb R)$.
\end{itemize}
Let $Y$ be a Banach space in one of the following situations:
\begin{enumerate}
\setlength\itemsep{0.3em}
\item $Y$ has the RNP and $\strexp{B_Y}$ is either countable up to rotations or discrete up to rotations.
\item $Y^*$ has the RNP and $\strexp{B_{Y^*}}$ is countable up to rotations.
\item $Y^*$ has the RNP and for every sequence $\{y_n^*\}$ of elements of $\wstrexp{B_{Y^*}}$ which converges to an element $y_0^*\in \strexp{B_{Y^*}}$, there exist $n_0\in \N$ and a sequence $\{\theta_n\}$ in $\T$ such that $y_n^*=\theta_n y_0^*$ for every $n\geq n_0$.
\end{enumerate}
Then, $\SA(M,Y)$ is dense in $\Lip(M,Y)$.
\end{corollary}

\begin{proof}
Observe that, by Theorem \ref{thm:LURrenorming} in the case (a) and by \cite[Theorem 3.7]{cgmr2020} in case (b), $\SE(\mathcal F(M))$ is dense in $\Lip(M,\mathbb R)$. Consequently, depending on the assumptions on the space $Y$, Corollary \ref{coro:countably-1}, \ref{coro:countably-2}, Theorem \ref{thm:RNP} or \ref{thm:dualRNP} concludes the result.
\end{proof}

\subsection{Bilinear forms and tensor product spaces}

In this subsection, we will give applications to the study of norm attaining bilinear forms.

Given Banach spaces $X, Y$ and $Z$, we let the notation $\mathcal{B} (X \times Y, Z)$ stand for the space of all continuous bilinear mappings from $X \times Y$ to $Z$ endowed with the norm $$\|B\| = \sup\{ \|B(x,y)\| \colon x \in B_X, \, y \in B_Y\}$$ for every $B\in \mathcal{B} (X \times Y, Z)$. When $Z =\mathbb{K}$, we simply denote the space by $\mathcal{B} (X \times Y)$. A bilinear mapping $B \in \mathcal{B} (X \times Y, Z)$ is said to be \emph{norm attaining} if the supremum defining $\|B\|$ is actually a maximum.  Let us denote by $\NA \mathcal{B} (X \times Y, Z)$ the set of all norm attaining bilinear mappings in $\mathcal{B} (X \times Y, Z)$.

Before exhibiting classical results in theory of denseness of bilinear mappings, let us exhibit the strong connection with norm attainment of bounded operators, for which we need to explain the useful language of tensor product spaces. The projective tensor product of $X$ and $Y$, denoted by $X \pten Y$, is the completion of the space $X \otimes Y$ endowed with the norm given by
\begin{eqnarray*}
\|z\|_{\pi} &=& \inf \left\{ \sum_{n=1}^{\infty} \|x_n\| \|y_n\|\colon \sum_{n=1}^{\infty} \|x_n\| \|y_n\| < \infty, z = \sum_{n=1}^{\infty} x_n \otimes y_n \right\} \\
&=& \inf \left\{ \sum_{n=1}^{\infty} |\lambda_n|\colon z = \sum_{n=1}^{\infty} \lambda_n x_n \otimes y_n, \sum_{n=1}^{\infty} |\lambda_n| < \infty, \|x_n\| = \|y_n\| = 1 \right\},
\end{eqnarray*}
where the infimum is taken over all such representations of $z$. It is well-known that $\|x \otimes y\|_{\pi} = \|x\| \|y\|$ for every $x \in X$, $y \in Y$, and the closed unit ball of $X \pten Y$ is the closed convex hull of the set $B_X \otimes B_Y = \{ x \otimes y\colon x \in B_X, y \in B_Y \}$. We refer the reader to \cite{ryan} for background on tensor product theory.

It is known that the three spaces $\mathcal B(X\times Y)$, $\Lin(X,Y^*)$ and $(X\pten Y)^*$ are isometrically isomorphic. Given $B\in \mathcal B(X\times Y)$, then $B$ can be seen as an operator $T_B\colon X\longrightarrow Y^*$ acting as $T_B(x)(y):=B(x,y)$. Moreover, $B$ can be seen as a linear functional acting on $X\pten Y$ as follows:
$$B\left(\sum_{i=1}^n x_i\otimes y_i \right):=\sum_{i=1}^n B(x_i,y_i).$$
Consequently, given $B\in \mathcal B(X\times Y)$ we have three different ways in which we can consider that $B$ is norm attaining: if $B$ is norm-attaining as a bilinear mapping, as an operator in $\Lin(X,Y^*)$, and as a functional on $X\pten Y$. Among all of them, the strongest notion is the one inherited from $\mathcal B(X\times Y)$.

On the one hand, if $B\in\mathcal B(X\times Y)$ then $B$ attains its norm as bilinear form if, and only if, $B$ attains its norm as bounded linear functionals on $X\pten Y$ at a point of the form $z=x\otimes y\in S_X\otimes S_Y$. Consequently, every bilinear form which attains its norm as bilinear form is a norm-attaining linear functional on $X\pten Y$. Let us mention that the converse is not true (see e.g.\ \cite[Remark 5.7]{Godefroy-survey-2015}).

On the other hand, if $B\in\mathcal B(X\times Y)$, then $B$ attains its norm as bilinear form at $(x,y)\in S_X\times S_Y$ if, and only if, the associated operator $T_B\colon X\longrightarrow Y^*$ attains its operator norm at $x$ and satisfies that $T_B(x)\in \NA(Y,\mathbb K)$. Let us also mention that the norm attaining of $B$ as bilinear form and the one of $T_B$ as operator are different, as also are the denseness associated to these two notions of norm attainment, see \cite{Choi1997,FinetPaya}.

In view of the previous connection between the different notion of norm attainment for a bilinear mapping, it is natural that sufficient conditions for the density of norm attaining bounded operators are behind most of the results in the literature about density of norm attaining bilinear mappings. Let us mention, for instance, that if $X$ has property $\alpha$ \cite{PS} (or even quasi-$\alpha$ \cite{CS}), then $\NA \mathcal{B} (X \times Y)$ is dense in $\mathcal{B} (X \times Y)$ with no assumption on $Y$. Moreover, it is observed in \cite{CS} that if $X$ has property quasi-$\alpha$ and $Y$ has property A, then actually $X \pten Y$ has property A. Recall that a Banach space $X$ has \textit{property quasi-$\alpha$} \cite{CS} if there exist a subset $\{x_\lambda\}_{\lambda \in \Lambda}$ of $S_X$, a subset $\{x_\lambda^*\}_{\lambda \in \Lambda}\subseteq S_{X^*}$, and $\rho \colon \Lambda \longrightarrow \mathbb{R}$ such that
			\begin{itemize}
				\item[(a)] $x^*_\lambda (x_\lambda)=1 $ for all $\lambda \in \Lambda$.
				\item[(b)] $|x^*_\lambda(x_\mu)| \leq \rho(\mu)<1$ for all $x_\lambda \neq x_\mu$.
				\item[(c)] For every $e^{**} \in \ext{B_{X^{**}}}$, there exists a subset $A_{e^{**}} \subseteq A$ such that $e^{**}$ belong to $\overline{A_{e^{**}}}^{\omega^*}$ and $r_{e^{**}}=\sup\{\rho(\mu)\colon x_\mu \in A_{e^{**}}\}<1$.
			\end{itemize}
If there is $0<R<1$ such that $r_{e^{**}}\leq R$ for every $e^{**}\in \ext{B_{X^{**}}}$, then the space $X$ has \emph{property $\alpha$}.

We can obtain stronger results than the mere denseness of $\NA \mathcal B(X\times Y)$ by using Theorem~\ref{thm:LURrenorming}. Let us say that a bilinear mapping $B \in \mathcal{B} (X \times Y, Z)$ is a \emph{strongly norm attaining bilinear mapping} if there exists $(x_0, y_0) \in B_X \times B_Y$ such that whenever a sequence $\{(x_n, y_n)\} \subset B_X \times B_Y$ satisfies $\| B(x_n, y_n)\| \longrightarrow \|B\|$, then there exists a subsequence $\{x_{k_n}, y_{k_n}\}$ such that $\{x_{k_n}\}$ and $\{y_{k_n}\}$  converge to $\alpha x_0$ and $\beta y_0$ for some $\alpha, \beta \in \mathbb{K}$ with $|\alpha| =|\beta| =1$, respectively.

Now we have the following result.

\begin{corollary}\label{cor:ChoiSong}
Let $X$ and $Y$ be Banach spaces. Suppose that $X$ has property quasi-$\alpha$ and $Y$ has property A.
Suppose that one of the following conditions holds:
\begin{enumerate}
\setlength\itemsep{0.3em}
\item $X$ and $Y$ both are separable,
\item either $X$ or $Y$ has the Dunford-Pettis property, and both $X$ and $Y$ are WCG spaces.
\end{enumerate}
Then, the set of strongly norm attaining bilinear forms in $\mathcal{B} (X \times Y)$ is dense in $\mathcal{B} (X \times Y)$.
\end{corollary}

\begin{proof}
Notice that, in any case, $X \pten Y$ is a WCG space (in the case (2), it follows from \cite[Theorem 16]{Diestel}). Moreover, $X \pten Y$ has property A as $X$ has property quasi-$\alpha$ and $Y$ has property A \cite{CS}. It follows from Corollary \ref{corollary:PropertyAimpliesSEdense} that $\SE (X \pten Y)$ is dense in $(X \pten Y)^* = \mathcal{B} (X \times Y)$. Suppose that $B \in (X \pten Y)^*$ strongly exposes $B_{X \pten Y}$ at some $\mu \in B_{X\pten Y}$. By \cite{Werner}, we have that $\mu = x_0 \otimes y_0$ for some $x_0 \in \strexp{B_X}$ and $y_0 \in \strexp{B_Y}$. Now, if a sequence $\{(x_n, y_n)\} \subset B_X \times B_Y$ satisfies $| B(x_n, y_n) | \longrightarrow \|B\|$, then $\{(x_n \otimes y_n)\} \longrightarrow ( \theta x_0) \otimes y_0$ for some $\theta \in \mathbb{K}$ with $|\theta|=1$. From this, we have that there are subsequences $\{x_{k_n}\}$ and $\{y_{k_n}\}$ such that $\{x_{k_n}\}$ converges to $\alpha x_0$ and $\{y_{k_n}\}$ converges to $\beta y_0$ for some $\alpha, \beta \in \mathbb{K}$ with $|\alpha| =|\beta| =1$.
\end{proof}

It is known that if $X$ and $Y$ are Banach spaces having RNP, then $\NA \mathcal{B} (X \times Y, Z)$ is dense in $\mathcal{B} (X \times Y, Z)$ for every Banach space $Z$ \cite{AFW} (compare this with the fact that there exists a Banach space $E$ with RNP such that $E \pten E$ fails to have RNP \cite{BourgainPisier}). This, in particular, shows that $X \pten Y$ has property A provided $X$ and $Y$ have the RNP. Thus, the same argument as in the proof of Corollary \ref{cor:ChoiSong} yields the following result.

\begin{corollary}\label{cor:bilinearRNP}
Let $X$ and $Y$ be Banach space. Suppose that $X$ and $Y$ have the RNP and one of the following conditions holds:
\begin{enumerate}
\setlength\itemsep{0.3em}
\item $X$ and $Y$ both are separable,
\item either $X$ or $Y$ has the Dunford-Pettis property, and both $X$ and $Y$ are WCG spaces.
\end{enumerate}
Then, the set of strongly norm attaining bilinear forms in $\mathcal{B} (X \times Y)$ is dense in $\mathcal{B} (X \times Y)$.
\end{corollary}

\vspace{1cm}
{\small \textbf{Acknowledgements:} This paper was partially written when the first author was visiting the University of Granada and he would like to acknowledge the hospitality that he received there. The authors would like to thank Antonio Avil\'{e}s, Luis Carlos Garc\'{\i}a-Lirola, Gilles Godefroy, Manuel Maestre, Vicente Montesinos, and Rafael Pay\'a for kindly answering several inquiries related to the topics of the paper.}

{\small M.~Jung was supported by NRF (NRF-2019R1A2C1003857), by POSTECH Basic Science Research Institute Grant (NRF-2021R1A6A1A10042944) and by a KIAS Individual Grant (MG086601) at Korea Institute for Advanced Study. M.~Mart\'in was supported by Spanish AEI Project PGC2018-093794-B-I00/AEI/10.13039/501100011033 (MCIU/AEI/FEDER, UE), by Junta de Andaluc\'ia I+D+i grants P20\_00255, A-FQM-484-UGR18, and FQM-185, and by ``Maria de Maeztu'' Excellence Unit IMAG, reference CEX2020-001105-M funded by MCIN/AEI/10.13039/501100011033. A.~Rueda Zoca was supported by MTM2017-86182-P (Government of Spain, AEI/FEDER, EU), by Spanish AEI Project PGC2018-093794-B-I00/AEI/ 10.13039/501100011033 (MCIU/AEI/FEDER, UE), by Fundaci\'on S\'eneca, ACyT Regi\'on de Murcia grant 20797/PI/18, by Junta de Andaluc\'ia Grant A-FQM-484-UGR18, and by Junta de Andaluc\'ia Grant FQM-0185.}

\end{document}